\RequirePackage{etoolbox}
\newtoggle{arxiv}
\toggletrue{arxiv} % if displaying the arxiv version
\iftoggle{arxiv}{
	\documentclass[runningheads, envcountsame]{llncs}
}{
	\documentclass{llncs}
}

\newenvironment{proofof}[1]{\noindent\emph{Proof of #1.}}{}

\usepackage{cite}
\usepackage[UKenglish]{babel}
\usepackage{cmap}
\usepackage{hyperref}

\usepackage{microtype}

\usepackage{amsmath,mathtools,amssymb}
\usepackage{siunitx}

\usepackage{enumitem}
\usepackage{booktabs}

\iftoggle{arxiv}{
	\hypersetup{
	    colorlinks = true,
	    citecolor=blue,
	}
}{}

\title{Computing Inferences for Large-Scale Continuous-Time Markov Chains by Combining Lumping with Imprecision}
\titlerunning{Continuous-Time Markov Chains: Combining Lumping with Imprecision}
\author{Alexander Erreygers \and Jasper De Bock}
\institute{
	Ghent University, ELIS, SYSTeMS \\
	\email{\string{alexander.erreygers, jasper.debock\string}@ugent.be}
}

\newcommand{\reals}{\bbbr}
\newcommand{\posreals}{\reals_{> 0}}
\newcommand{\nnegreals}{\reals_{\geq 0}}
\newcommand{\nats}{\bbbn}
\newcommand{\natz}{\bbbn_{0}}

\newcommand{\indica}[1]{\bbbone_{#1}}

\DeclarePairedDelimiter{\abs}{\lvert}{\rvert}
\DeclarePairedDelimiter{\card}{\lvert}{\rvert}
\DeclarePairedDelimiter{\norm}{\lVert}{\rVert}
\DeclarePairedDelimiterXPP{\varnorm}[1]{}{\lVert}{\rVert}{_{v}}{#1}

\newcommand\condbase[1][]{\:#1\lvert\:}
\let\cond\condbase

\DeclarePairedDelimiterX\gr[1](){\let\cond\scond #1}

\providecommand\given{}
\DeclarePairedDelimiterX{\set}[1]{\{}{\}}{%
	\renewcommand{\given}{\colon}
	#1
}

\newcommand{\setoffn}{\mathcal{L}}

\newcommand{\prob}{P}
\newcommand{\markov}{P}

\newcommand{\probm}{\prob}
\newcommand{\problum}{\lu{\prob}}
\newcommand{\prev}{E}
\newcommand{\prevm}{\prev}
\newcommand{\prevlum}{\lu\prev}
\newcommand{\limprevm}{\prev_{\infty}}

\newcommand{\setofconsproc}{\mathbb{P}}

\newcommand{\stsp}{\mathcal{X}}

\newcommand{\setofposdist}{\mathcal{D}}

\newcommand{\setoftseq}{\mathcal{U}}
\newcommand{\setoftseqne}{\mathcal{U}_{\emptyset}}
\newcommand{\setoftseql}[1]{\mathcal{U}_{< #1}}
\newcommand{\setoftseqnel}[1]{\mathcal{U}_{\emptyset, < #1}}

\newcommand{\emptytseq}{\emptyset}
\newcommand{\emptytup}{\diamond}

\newcommand{\lu}[1]{\hat{#1}}
\newcommand{\lustsp}{\lu{\stsp}}
\newcommand{\lumap}{\Lambda}
\newcommand{\invlumap}{\Gamma}

\newcommand{\dist}{\pi}
\newcommand{\inidist}{\dist_{0}}
\newcommand{\distst}{\dist^{\star}}
\newcommand{\luinidist}{\lu{\dist}_{0}}
\newcommand{\limdist}{\dist_{\infty}}

\newcommand{\lulimdist}{\lu{\dist}_{\infty}}
\newcommand{\credset}{\mathcal{M}}

\newcommand{\trm}{Q}

\newcommand{\lutrm}{\lu{\trm}}
\newcommand{\lutrmst}{\lu{\trm}^{\star}}
\newcommand{\setofalltrm}{\mathcal{R}}
\newcommand{\setoftrm}{\mathcal{\trm}}
\newcommand{\setoflutrm}{\lu{\setoftrm}}
\newcommand{\ltro}{\underline{\trm}}
\newcommand{\lultro}{\underline{\lutrm}}

\newcommand{\tm}{T}

\newcommand{\lto}{\underline{T}}
\newcommand{\lulto}{\underline{\lu{\tm}}}

\newcommand{\acces}{\rightsquigarrow}
\newcommand{\upacces}{\rightarrowtail}
\newcommand{\paths}{\Omega}
\newcommand{\pad}{\omega}
\newcommand{\elementaryevents}{\mathcal{E}}
\newcommand{\eventalgebra}{\mathcal{A}}
\newcommand{\ccpdomain}{\mathcal{C}}
\newcommand{\ccpevents}{\mathcal{E}}
\newcommand{\procdomain}{\mathcal{C}^{\mathrm{SP}}}
\newcommand{\luprocdomain}{\lu{\mathcal{C}}^{\mathrm{SP}}}

\spnewtheorem{definition}{Definition}{\bfseries}{\rmfamily}

\begin{document}

\maketitle

\begin{abstract}
	If the state space of a homogeneous continuous-time Markov chain is too large, making inferences---here limited to determining marginal or limit expectations---becomes computationally infeasible.
	Fortunately, the state space of such a chain is usually too detailed for the inferences we are interested in, in the sense that a less detailed---smaller---state space suffices to unambiguously formalise the inference.
	However, in general this so-called lumped state space inhibits computing exact inferences because the corresponding dynamics are unknown and/or intractable to obtain.
	We address this issue by considering an imprecise continuous-time Markov chain.
	In this way, we are able to provide guaranteed lower and upper bounds for the inferences of interest, without suffering from the curse of dimensionality.
\end{abstract}
% \keywords{(Imprecise) continuous-time Markov chain \and Lumping \and State space explosion.}

	\section{Introduction}
\label{sec:Introduction}
\emph{State space explosion}, or the exponential dependency of the size of a finite state space on a system's dimensions, is a frequently encountered inconvenience when constructing mathematical models of systems.
In the setting of continuous-time Markov chains (CTMCs), this exponentially increasing number of states has as a consequence that using the model to perform inferences---for the sake of brevity here limited to marginal and limit expectations---about large-scale systems becomes computationally intractable.
Fortunately, for many of the inferences we would like to make, a higher-level state description actually suffices, allowing for a reduced state space with considerably fewer states.
However, unfortunately, the low-level description and its corresponding larger state space are necessary in order to accurately model the system's dynamics.
Therefore, using the reduced state space to make inferences is generally impossible.

In this contribution, we address this problem using imprecise continuous-time Markov chains \cite{2017KrakDeBock,2016DeBock,2015Skulj}.
In particular, we outline an approach to determine guaranteed lower and upper bounds on marginal and limit expectations using the reduced state space.
We introduced a preliminary version of this approach in \cite{2017Rottondi,2018Erreygers}, but the current contribution is---to the best of our knowledge---its first fully general and theoretically justified exposition.
Compared to other approaches \cite{1994FraceschinisMuntz,2005Buchholz} that also determine lower and upper bounds on expectations, ours has the advantage that it is not restricted to limit expectations.
Furthermore, based on our preliminary experiments, our approach seems to produce tighter bounds.

% We argue how lumping naturally induces an imprecise CTMC in Sect.~\ref{sec:Lumping and the lumped stochastic processs}, and how we can use this to obtain lower and upper bounds on conditional expectations with respect to the lumped process.
% In Sect.~\ref{sec:Computing bounds on limit expectations}, we use the theory of imprecise CTMCs to provide guaranteed lower and upper bounds on limit expectations with respect to both the original as the lumped process.

	\section{Continuous-Time Markov Chains}
\label{sec:CTMCs}
We are interested in making inferences about a system, more specifically about the state of this system at some future time $t$, denoted by $X_{t}$.
The complication is that we are unable to predict the temporal evolution of the state with certainty.
Therefore, at all times $t \in \nnegreals{}$,\footnote{%
	We use $\nnegreals$ and $\posreals$ to denote the set of non-negative real numbers and positive real numbers, respectively.
	Furthermore, we use $\nats$ to denote the natural numbers and write $\natz$ when including zero.
} the state~$X_{t}$ of the system is a random variable that takes values---generically denoted by $x$, $y$ or $z$---in the state space $\stsp$.

	\subsection{Homogeneous Continuous-Time Markov Chains}
\label{ssec:Homogeneous CTMCs}
We assume that the stochastic process that models our beliefs about the system, denoted by $\gr{X_t}_{t \in \nnegreals}$, is a \emph{continuous-time Markov chain} (CTMC) that is \emph{homogeneous}.
For a thorough treatment of the terminology and notation concerning CTMCs, we refer to \cite{1997Norris,1991Anderson,2017KrakDeBock}.
Due to length constraints, we here limit ourselves to the bare necessities.

The stochastic process $\gr{X_t}_{t \in \nnegreals}$ is a CTMC if it satisfies the \emph{Markov property}, which says that for all $t_1, \dots, t_n, t, \Delta$ in $\nnegreals$ with $n\in\nats$ and $t_1 < \cdots < t_n < t$, and all $x_1, \dots, x_n, x, y$ in $\stsp$,
\begin{equation}
\label{eqn:HCTMC Markov}
	\probm\gr{X_{t+\Delta} = y \cond X_{t_1} = x_1 \dots, X_{t_n} = x_{n}, X_{t} = x}
	= \probm\gr{X_{t+\Delta} = y \cond X_{t} = x}.
\end{equation}
The CTMC $\gr{X_t}_{t \in \nnegreals{}}$ is \emph{homogeneous} if for all $t, \Delta$ in $\nnegreals$ and all $x, y$ in $\stsp$,
\begin{equation}
\label{eqn:HCTMC Homogeneity}
	\probm\gr{X_{t+\Delta} = y \cond X_t = x}
	= \probm\gr{X_{\Delta} = y \cond X_0 = x}.
\end{equation}

It is well-known that---both in the classical measure-theoretic framework \cite{1991Anderson} and the full conditional framework \cite{2017KrakDeBock}---a homogeneous continuous-time Markov chain is uniquely characterised by a triplet $(\stsp, \inidist, \trm)$, where $\stsp$ is a state space, $\inidist$ an initial distribution and $\trm$ a transition rate matrix.
% In the remainder, we will use the notation $\markov \coloneqq \gr{\stsp, \inidist{}, \trm}$ to denote the homogeneous continuous-time Markov chain characterised by this triplet.

The state space~$\stsp$ is taken to be a non-empty, finite and---without loss of generality---ordered set.
This way, any real-valued function~$f$ on $\stsp$ can be identified with a column vector, the $x$-component of which is $f(x)$.
The set containing all real-valued functions on $\stsp$ is denoted by $\setoffn\gr{\stsp}$.

The initial distribution~$\inidist$ is defined by
\begin{equation}
\label{eqn:HCTMC initial}
	\inidist\gr{x}
	\coloneqq \probm\gr{X_0 = x}
	\text{ for all $x$ in $\stsp$,}
\end{equation}
and hence is a probability mass function on $\stsp$.
We will (almost) exclusively be concerned with positive (initial) distributions, whom we collect in $\setofposdist\gr{\stsp}$ and will identify with row vectors.

The transition rate matrix~$\trm$ is a real-valued $\card{\stsp} \times \card{\stsp}$ matrix---or equivalently, a linear map from $\setoffn\gr{\stsp}$ to $\setoffn\gr{\stsp}$---with non-negative off-diagonal entries and rows that sum up to zero.
If for any $t$ in $\nnegreals$ we define the \emph{transition matrix over $t$} as
\begin{equation}
\label{eqn:Matrix exponential}
	\tm_t
	\coloneqq e^{t Q}
	= \lim_{n \to +\infty} \gr*{I + \frac{t}{n} Q}^{n},
	% = \sum_{n=0}^{+\infty} \frac{\gr{\Delta \trm}^{n}}{n!}
\end{equation}
% As is evident from the name and is well-known---see for instance \cite[Theorem~2.1.2]{1997Norris}---$\tm_{\Delta}$ is a transition matrix, in the sense that it is a real-valued $\card{\stsp} \times \card{\stsp}$ matrix with non-negative entries and rows that sum up to one.
then for all $t$ in $\nnegreals$ and all $x, y$ in $\stsp$,
\begin{equation}
\label{eqn:HCTMC transition}
	\probm\gr{X_t = y \cond X_0 = x}
	= \tm_t\gr{x, y}.
\end{equation}

Finally, we denote by $\prevm$ the expectation operator with respect to the homogeneous CTMC $\gr{X_t}_{t \in\nnegreals}$ in the usual sense.
It follows immediately from \eqref{eqn:HCTMC initial} and \eqref{eqn:HCTMC transition} that
$
	\prevm\gr{f\gr{X_t}}
	% \coloneqq{}& \sum_{y \in \stsp} \probm\gr{X_{t+\Delta} = y \cond X_u = x_u, X_t = x} f\gr{y} \\
	= \pi_0 \tm_t f
$ for any $f$ in $\setoffn\gr{\stsp}$ and any $t$ in $\nnegreals$.

	\subsection{Irreducibility}
\label{ssec:Irreducibility of homogeneous CTMCs}
In order not to be tangled up in edge cases, in the remainder we are only concerned with irreducible transition rate matrices.
Many equivalent necessary and sufficient conditions exist; see for instance \cite[Theorem~3.2.1]{1997Norris}.
For the sake of brevity, we here say that a transition rate matrix $\trm$ is \emph{irreducible} if, for all $t$ in $\posreals$ and $x, y$ in $\stsp$, $\tm_t \gr{x, y} > 0$.
% \begin{definition}
% \label{def:Irreducibility}
% 	The transition rate matrix $\trm$ is called \emph{irreducible} if, for all $t$ in $\posreals$ and $x, y$ in $\stsp$, $\tm_t \gr{x, y} > 0$.
% \end{definition}
% A generic transition rate matrix $\trm$ is called \emph{irreducible} if, for all two states $x, y$ in $\stsp$ and all positive real numbers $\Delta$, $\tm_{\Delta}\gr{x, y} > 0$.

Consider now a homogeneous CTMC that is characterised by $(\stsp, \inidist, \trm)$.
It is then well-known that for any $f$ in $\setoffn\gr{\stsp}$, the limit $\lim_{t \to +\infty} \prevm\gr{f\gr{X_t}}$ exists.
Even more, since we assume that $\trm$ is irreducible, this limit value is the same for all initial distributions $\inidist$ \cite[Theorem~3.6.2]{1997Norris}!
This common limit value, denoted by $\limprevm\gr{f}$, is called the \emph{limit expectation of $f$}.
Furthermore, the irreducibility of $\trm$ also implies that there is a unique stationary distribution $\limdist$ in $\setofposdist\gr{\stsp}$ that satisfies the \emph{equilibrium condition} $\limdist \trm = 0$.
This unique distribution is called the \emph{limit distribution}, as
$
	\limprevm\gr{f}
	= \limdist f
$.
% Solving the equilibrium condition is usually a

% It is obvious that in this form, checking irreducibility is not trivial as it requires the evaluation of \eqref{eqn:Matrix exponential} for all $\Delta$.
% % However, there are equivalent necessary and sufficient conditions for irreducibility that are easier to check.
% One of the more convenient equivalent conditions only requires us to check the sign of $\trm\gr{x, y}$ for certain combinations of distinct $x$ and $y$, see ** cite **.
In the remainder of this contribution, a \emph{positive and irreducible CTMC} is any homogeneous CTMC characterised by a positive initial distribution~$\inidist$ and an irreducible transition rate matrix~$\trm$.

	\section{Lumping and the Induced (Imprecise) Process}
\label{sec:Lumping and the lumped stochastic processs}
In many practical applications---see for instance \cite{2017Rottondi,2018Erreygers,2005Buchholz,1994FraceschinisMuntz}---we have a positive and irreducible CTMC that models our system and we want to use this chain to make inferences of the form $\prevm\gr{f\gr{X_t}} = \inidist \tm_t f$ or $\limprevm\gr{f}$.
As analytically evaluating the limit in \eqref{eqn:Matrix exponential} is often infeasible, we usually have to resort to one of the many available numerical methods---see for example \cite{2003MolerVanLoan}---that approximate $\tm_t$.
However, unfortunately these numerical methods turn out to be computationally intractable when the state space becomes large.
Similarly, determining the unique distribution $\limdist$ that satisfies the equilibrium condition also becomes intractable for large state spaces.

	% \subsection{The Reduced State Space}
Fortunately, as previously mentioned in Sect.~\ref{sec:Introduction}, the state space~$\stsp$ is often unnecessarily detailed.
Indeed, many interesting inferences can usually still be unambiguously defined using real-valued functions on a less detailed state space that corresponds to a higher-order description of the system, denoted by $\lustsp$.
However, this provides no immediate solution as the motive behind using the detailed state space~$\stsp$ in the first place is that this allows us to accurately model the (uncertain) dynamics of the system using a homogeneous CTMC; see \cite{2017Rottondi,2018Erreygers,2005Buchholz,1994FraceschinisMuntz,2014Ganguly} for practical examples.
In contrast, the dynamics of the induced stochastic process on the the reduced state space~$\lustsp$ are often unknown and/or intractable to obtain, which inhibits us from making exact inferences using the induced stochastic process.
We now set out to address this by allowing for imprecision.

	\subsection{Notation and Terminology Concerning Lumping}
We assume that the lumped state space~$\lustsp$ is obtained by \emph{lumping}---sometimes called grouping or aggregating, see \cite{1958Burke,1993BallYeo}---states in $\stsp$, such that $1 < \card{\lustsp} \leq \card{\stsp{}}$.
This lumping is formalised by the surjective \emph{lumping map} $\lumap \colon \stsp \to \lustsp{}$, which maps every state $x$ in $\stsp$ to a state $\lumap\gr{x} = \lu{x}$ in $\lustsp{}$.
In the remainder, we also use the inverse lumping map $\invlumap$, which maps every $\lu{x}$ in $\lustsp$ to a subset
\(
	\invlumap\gr{\lu{x}}
	\coloneqq \set*{ x \in \stsp \given \lumap\gr{x} = \lu{x} }
	% \text{ for all $\lu{x}$ in $\lustsp$}.
\) of \(\stsp\).
Given such a lumping map $\lumap$, a function $f$ in $\setoffn\gr{\stsp}$ is \emph{lumpable with respect to $\lumap$} if
%it is constant on the ``lumps'', in the sense that
there is an $\lu{f}$ in $\setoffn\gr{\lustsp}$ such that $f\gr{x} = \lu{f}\gr{\lumap\gr{x}}$ for all $x$ in $\stsp$.
We use $\setoffn_{\lumap}\gr{\stsp} \subseteq \setoffn\gr{\stsp}$ to denote the set of all real-valued functions on $\stsp$ that are lumpable with respect to $\lumap$.

As far as our results are concerned, it does not matter in which way the states are lumped.
For a given $f$ in $\setoffn\gr{\stsp}$---recall that we are interested in the (limit) expectation of $f(X_t)$---a naive choice is to lump together all states that have the same image under $f$.
However, this is not necessarily a good choice.
One reason is that the resulting lumped state space can become very small, for example when $f$ is an indicator, resulting in too much imprecision in the dynamics and/or the inference.
Lumping-based methods therefore often let $\lustsp$ correspond to a natural higher-level description of the state of the system; see for example \cite{2018Erreygers,2005Buchholz,1994FraceschinisMuntz} for some positive results.
An extra benefit of this approach is that the resulting model can be used to determine the (limit) expectation of multiple functions.
% For example, using such a lumping \num{8008} states are reduced to only \num{125} states for the largest system reported in \cite[Tab.~3]{1994FraceschinisMuntz}, while \num{1221759} states are reduced to \num{35301} lumps for the largest system reported in \cite[Tab.~1]{2018Erreygers}.

	\subsection{The Lumped Stochastic Process}
Let \smash{$\gr{X_t}_{t \in \nnegreals{}}$} be a positive and homogeneous continuous-time Markov chain.
% Then any lumping map $\lumap \colon \stsp \to \lustsp$ unequivocally induces the \emph{lumped stochastic process} that has $\lustsp$ as state space:
% \[
% 	\gr{\lu{X}_t}_{t \in \nnegreals{}}
% 	\coloneqq \gr{\lumap\gr{X_t}}_{t \in \nnegreals{}}.
% \]
Then any lumping map \smash{$\lumap \colon \stsp \to \lustsp$} unequivocally induces a \emph{lumped stochastic process}~$\gr{\lu{X}_t}_{t \in \nnegreals{}}$.
It has $\lustsp$ as state space and is defined by the relation
\begin{equation}
\label{eqn:Naive definition of lumped process}
	\gr{\lu{X}_t = \lu{x}}
	\Leftrightarrow \gr{X_t \in \invlumap\gr{\lu{x}}}
	\text{ for all $t$ in $\nnegreals$ and all $\lu{x}$ in $\lustsp$}.
\end{equation}

In some cases, this lumped stochastic process is a homogeneous CTMC, and the inference of interest can then be computed using this reduced CTMC.
See for example \cite[Theorem~2.3(i)]{1993BallYeo} for a necessary condition and \cite[Theorem~2.4]{1993BallYeo} or \cite[Theorem~3]{1958Burke} for a necessary and sufficient one.
However, these conditions are very stringent.
Indeed, in general, the lumped stochastic process is not homogeneous nor Markov.
For this general case, we are not aware of any previous work that characterises the dynamics of the lumped stochastic process efficiently---i.e., directly from $\lumap$, $\trm$ and $\inidist$ and without ever determining $\tm_t$.

	\subsection{The Induced Imprecise Continuous-Time Markov Chain}
\label{ssec:Induced imprecise CTMC}
Nevertheless, that is exactly what we now set out to do.
\iftoggle{arxiv}{
	We here only provide an intuitive explanation of our methodology; for a detailed exposition, we refer to Appendix~\ref{app:Induced iCTMC}.
}{
	Due to length constraints, we will here restrict ourselves to providing an intuitive explanation of our methodology, becoming formal only when stating our main results; see Theorems~\ref{the:Bounds on marginal expectation} and \ref{the:iterative approximation} further on.
	For a detailed exposition, we refer to the appendix of the extended preprint of this contribution \cite{extended}.
}

The essential point is that, while we cannot exactly determine the dynamics of the lumped stochastic process~$\gr{\lu{X}_t}_{t \in \nnegreals}$, we can consider a \emph{set of possible stochastic processes}, not necessarily homogeneous and/or Markovian but all with \smash{$\lustsp$} as state space, that definitely contains the lumped stochastic process~$\gr{\lu{X}_t}_{t \in \nnegreals}$.
In the remainder, we will denote this set by $\setofconsproc_{\inidist, \trm, \lumap}$.
As is indicated by our notation, $\setofconsproc_{\inidist, \trm, \lumap}$ is fully characterised by $\inidist$, $\trm$ and $\lumap$.

Crucially, it turns out that $\setofconsproc_{\inidist, \trm, \lumap}$ takes the form of a so-called \emph{imprecise continuous-time Markov chain}.
For a formal definition of general imprecise CTMCs, and an extensive study of their properties, we refer the reader to the work of Krak~et.~al.~\cite{2017KrakDeBock} and De Bock~\cite{2016DeBock}.
For our present purposes, it suffices to know that tight lower and upper bounds on the expectations that correspond to the set of stochastic processes of an imprecise CTMC are relatively easy to obtain.
In particular, they can be determined without having to explicitly optimise over this set of processes, thus mitigating the need to actually construct it.

There are many parallels between homogeneous CTMCs and imprecise CTMCs.
For instance, the counterpart of a transition rate matrix is a \emph{lower transition rate operator}.
For our imprecise CTMC~$\setofconsproc_{\inidist, \trm, \lumap}$, this lower transition rate operator is \smash{$\lultro \colon \setoffn\gr{\lustsp} \to \setoffn\gr{\lustsp} \colon g \mapsto \lultro g$} where, for every $g$ in \smash{$\setoffn\gr{\lustsp}$}, \smash{$\lultro g$} is defined by
\begin{equation}
\label{eqn:lumped ltro}
	[\lultro g]\gr{\lu{x}}
	\coloneqq \min \set*{ \sum_{\lu{y} \in \lustsp} g\gr{\lu{y}} \sum_{y \in \invlumap\gr{\lu{y}}} Q(x, y) \colon x \in \invlumap\gr{\lu{x}} }
	\text{ for all $\lu{x}$ in $\lustsp$}.
\end{equation}
Important to mention here is that in case the lumped state space corresponds to some higher-order state description, we often find that executing the optimisation in \eqref{eqn:lumped ltro} is fairly straightforward, as is for instance observed in \cite{2017Rottondi,2018Erreygers}.

The counterpart of the transition matrix over $t$ is now the \emph{lower transition operator over $t$}, denoted by $\lulto_t \colon \setoffn\gr{\lustsp} \to \setoffn\gr{\lustsp}$ and defined for all $g$ in $\setoffn\gr{\lustsp}$ by
\begin{equation}
\label{eqn:lumped lower transition operator}
	\lulto_t g
	\coloneqq \lim_{n \to +\infty} \gr*{I + \frac{t}{n} \lultro}^{n} g,
\end{equation}
where the $n$-th power should be interpreted as consecutively applying the operator $n$ times.
Note how strikingly \eqref{eqn:lumped lower transition operator} resembles \eqref{eqn:Matrix exponential}.
% Note that $\lulto_t \colon \setoffn\gr{\lustsp} \to \setoffn\gr{\lustsp}$ is a non-linear operator that is non-negatively homogeneous and that dominates the minimum.
Analogous to the precise case, one needs numerical methods---see for instance \cite{2017Erreygers} or \cite[Sect.~8.2]{2017KrakDeBock}---to approximate $\lulto_t g$ because analytically evaluating the limit in \eqref{eqn:lumped lower transition operator} is, at least in general, impossible.

	\section{Performing Inferences Using The Lumped Process}
\label{sec:Performing inferences using the lumped process}
\iftoggle{arxiv}{
	Everything is now set up to present our main results; see Appendix~\ref{app:Proofs of main results} for their proofs.
}{
	Everything is now set up to present our main results.
	Due to length constraints, we have relegated our proofs to the appendix of the extended arXiv version of this contribution \cite{extended}.
	% If we want to use conjugacy
	% Before stating our main results, the proofs of which we have relegated to the appendix of the extended arXiv version of this contribution \cite{extended} due to length constraints, we introduce two operators for notational convenience:
	% $\luutro \colon \setoffn\gr{\lustsp} \to \setoffn\gr{\lustsp}$ and $\luuto_t \colon \setoffn\gr{\lustsp} \to \setoffn\gr{\lustsp}$.
	% These are defined for all $t$ in $\nnegreals$ and $g$ in $\lustsp$ by the conjugacy relations $\luutro g \coloneqq - \lultro \gr{- g}$ and $\luuto_t g \coloneqq - \lulto_t \gr{- g}$.
}

	\subsection{Guaranteed Bounds On Marginal Expectations}
We first turn to marginal expectations.
Once we have $\setofconsproc_{\inidist, \trm, \lumap}$, the following result is a---not quite immediate---consequence of \cite[Corollary~8.3]{2017KrakDeBock}.
\begin{theorem}
\label{the:Bounds on marginal expectation}
	Consider a positive and irreducible CTMC characterised by $(\stsp, \inidist, \trm)$ and a lumping map $\lumap \colon \stsp \to \lustsp$.
	Let $f$ in $\setoffn\gr{\stsp}$ be lumpable with respect to $\lumap$ and let $\lu{f}$ be the corresponding element of $\setoffn\gr{\lustsp}$.
	Then for any $t$ in $\nnegreals$,
	\[
		\luinidist \lulto_t \lu{f}
		\leq \prev\gr{f\gr{X_t}}
		= \inidist \tm_t f
		\leq
		% \luinidist \luuto_t \lu{f} =
		- \luinidist \lulto_t \gr{- \lu{f}},
	\]
	where $\luinidist$ in $\setofposdist\gr{\lustsp}$ is defined by $\luinidist\gr{\lu{x}} \coloneqq \sum_{x \in \invlumap\gr{\lu{x}}} \inidist\gr{x}$ for all $\lu{x}$ in $\lustsp$.
\end{theorem}

This result is highly useful in the setting that was outlined in Sect.~\ref{sec:Lumping and the lumped stochastic processs}.
Indeed, for large systems we can use Theorem~\ref{the:Bounds on marginal expectation} to compute guaranteed lower and upper bounds on marginal expectations that cannot be computed exactly.

	\subsection{Guaranteed Bounds on Limit Expectations}
\label{ssec:Iterative computation of bounds on lumped limit expectation}
Our second result provides guaranteed lower and upper bounds on limit expectations.
This is extremely useful because the limit expectation is (almost surely) equal to the long-term temporal average due to the ergodic theorem \cite[Theorem~3.8.1]{1997Norris}, and in practice---see for instance \cite{2018Erreygers}---the inference one is interested in is often a long-term temporal average.
\begin{theorem}
\label{the:iterative approximation}
	Consider an irreducible CTMC and a lumping map $\lumap \colon \stsp \to \lustsp$.
	Let $f$ in $\setoffn\gr{\stsp}$ be lumpable with respect to $\lumap$ and let $\lu{f}$ be the corresponding element of $\setoffn\gr{\lustsp}$.
	Then for all $n$ in $\natz$ and $\delta$ in $\posreals$ such that $\delta \max \set*{ \abs{\trm\gr{x, x}} \colon x \in \stsp } < 1$,
	\begin{equation*}
		\min \gr{I + \delta \lultro}^{n} \lu{f}
		\leq \limprevm\gr{f}
		\leq
		% \max \gr{I + \delta \lultro}^n \lu{f} =
		- \min \gr{I + \delta \lultro}^{n} \gr{-\lu{f}}.
	\end{equation*}
	Furthermore, for fixed $\delta$, the lower and upper bounds in this expression become monotonously tighter with increasing $n$, and each converges to a (possibly different) constant as $n$ approaches $+ \infty$.
\end{theorem}

This result can be used to devise an approximation method similar to \cite[Algorithm~1]{2018Erreygers}: we fix some value for $\delta$, set $g_0 = \lu{f}$ (or $g_0 = - \lu{f}$) and then repeatedly compute $g_{i} \coloneqq \gr{I + \delta \lultro} g_{i-1} = g_{i-1} + \delta \ltro g_{i-1}$ until we empirically observe convergence of $\min g_i$ (or $- \min g_i$).
In general, the lower and upper bounds obtained in this way are dependent on the choice of $\delta$ and this choice can therefore influence the tightness of the obtained bounds.
Empirically, we have seen that smaller $\delta$ tend to yield tighter bounds, at the expense of requiring more iterations---that is, larger $n$---before empirical convergence.

\subsection{Some Preliminary Numerical Results}
Due to length constraints, we leave the numerical assessment of Theorem~\ref{the:Bounds on marginal expectation} for future work.
For an extensive numerical assessment of---the method implied by---Theorem~\ref{the:iterative approximation}, we refer the reader to \cite{2018Erreygers}.
% , where we used an algorithm that is very similar to that outlined in Sect.~\ref{ssec:Iterative computation of bounds on lumped limit expectation}.
We believe that in this contribution, it is more fitting to compare our method to the only existing method---at least the only one that we are aware of---that also uses lumping to provide guaranteed lower and upper bounds on limit expectations.
This method was first outlined by Franceschinis and Muntz~\cite{1994FraceschinisMuntz}, and then later improved by Buchholz~\cite{2005Buchholz}.
In order to display the benefit of their methods, they use them to determine bounds on several performance measures for a closed queueing network that consists of a single server in series with multiple parallel servers.
We use the method outlined in Sect.~\ref{ssec:Iterative computation of bounds on lumped limit expectation} to also compute bounds on these performance measures, as reported in Table~\ref{tab:Comparison of limits}.
Note that our bounds are tighter than those of \cite{1994FraceschinisMuntz}.
We would very much like to compare our method with the improved method of \cite{2005Buchholz} as well.
Unfortunately, the system parameters Buchholz uses do not---as far as we can tell---correspond to the number of states and the values for the performance measures he reports in \cite[Fig.~3]{2005Buchholz}, thus preventing us from comparing our results.

\begin{table}[th]
	\centering
	\caption{%
		Comparison of the bounds obtained by using Theorem~\ref{the:iterative approximation} with those obtained by the method presented in \cite[Sect.~3.2]{1994FraceschinisMuntz} for the closed queueing network of \cite{1994FraceschinisMuntz}.
	}
	\label{tab:Comparison of limits}
	\begin{tabular}{rS[table-format=2.5]S[table-format=2.5]S[table-format=2.5]S[table-format=2.5]S[table-format=2.5]}
		\toprule
			& & \multicolumn{2}{c}{\cite[Tab.~1]{1994FraceschinisMuntz}} & \multicolumn{2}{c}{Theorem~\ref{the:iterative approximation}} \\
		\cmidrule(lr){3-4}\cmidrule(l){5-6}
			 & {Exact} & {Lower} & {Upper} & {Lower} & {Upper} \\
		\cmidrule(lr){2-2}\cmidrule(lr){3-3}\cmidrule(lr){4-4}\cmidrule(lr){5-5}\cmidrule(l){6-6}
			Mean queue length & 1.2734 & 1.2507 & 1.3859 & 1.2664 & 1.2802 \\
			Throughput & 0.9828 & 0.9676 & 0.9835 & 0.9826 & 0.9831  \\
		\bottomrule
	\end{tabular}
\end{table}

	\section{Conclusion}
Broadly speaking, the conclusion of this contribution is that imprecise CTMCs are not only a robust uncertainty model---as they were originally intended to be---but also a useful computational tool for determining bounds on inferences for large-scale CTMCs.
More concretely, the first important observation of this contribution is that lumping states in a homogeneous CTMC inevitably introduces imprecision, in the sense that we cannot exactly determine the parameters that describe the dynamics of the lumped stochastic process without also explicitly determining the original process.
The second is that we can easily characterise a set of processes that definitely contains the lumped process, in the form of an imprecise CTMC.
Using this imprecise CTMC, we can then determine guaranteed lower and upper bounds on marginal and limit expectations with respect to the original chain.
From a practical point of view, these results are helpful in cases where state space explosion occurs: they allow us to determine guaranteed lower and upper bounds on inferences that we otherwise could not determine at all.

Regarding future work, we envision the following.
For starters, a more thorough numerical assessment of the methods outlined in Sect.~\ref{sec:Performing inferences using the lumped process} is necessary.
Furthermore, it would be of theoretical as well as practical interest to determine  bounds on the \emph{conditional} expectation of a lumpable function, or to consider functions that depend on the state at \emph{multiple} time points.
Finally, we are developing a method to determine lower and upper bounds on limit expectations that only requires the solution of a simple linear program.
% ; it would be interesting to compare its performance to the method of Sect.~\ref{ssec:Iterative computation of bounds on lumped limit expectation}.

\vspace{-.25em}
	\section*{Acknowledgements}
Jasper De Bock's research was partially funded by H2020-MSCA-ITN-2016 UTOPIAE, grant agreement 722734.
Furthermore, the authors are grateful to the reviewers for their constructive feedback and useful suggestions.
% Jasper De Bock's research was partially funded by the European Commission's H2020 programme, through the UTOPIAE Marie Curie Innovative Training Network, H2020-MSCA-ITN-2016, Grant Agreement number 722734.

\iftoggle{arxiv}{
	\renewcommand{\doi}[1]{doi:\href{https://doi.org/#1}{#1}}
	\providecommand{\arxid}[2]{arXiv:\href{https:arxiv.org/abs/#1}{#1} [#2]}
	\bibliography{biblio.bib}
}{
	% \vspace{-.25em}
	\providecommand{\url}[1]{\texttt{#1}}
	\providecommand{\urlprefix}{URL }
	\providecommand{\doi}[1]{https://doi.org/#1}
	\providecommand{\arxid}[2]{arXiv:\href{https:arxiv.org/abs/#1}{#1}}
	% \bibliography{biblio.bib}
	
}

\iftoggle{arxiv}{
\appendix

% 	\section*{Appendix}
% Coming soon

	\section{Some Notation}
We often rely on results from Krak et~al.~\cite{2017KrakDeBock} when proving our results.
In order to facilitate the use of these results, we try to adhere to the notation used in \cite{2017KrakDeBock} as much as possible.
For instance, throughout this appendix we will denote the original CTMC by $\prob$ instead of by $\gr{X_t}_{t \in \nnegreals{}}$, which we previously used in the main text.

	\subsection{Sequences of Time Points}
\label{sapp:Sequences of time points}
A finite number of time points $t_{1}, \dots, t_{n}$ in $\nnegreals$ is always taken to be increasing, in the sense that $t_{1} < \cdots < t_{n}$.
Following Krak et~al.~\cite{2017KrakDeBock}, we collect all such sequences---including the empty sequence $\emptytseq$---in the set $\setoftseq$, and denote a generic element of this set by $u$.
Furthermore, we denote the set of all time sequences without the empty sequence by $\setoftseqne$, and for all $t$ in $\nnegreals$ use $\setoftseql{t}$ (or $\setoftseqnel{t}$) to denote the set of all (non-empty) time sequences of which the last time point strictly precedes $t$.
Moreover, for any sequence $u = t_{1}, \dots, t_{n}$ in $\setoftseq$, we define $\stsp_u \coloneqq \prod_{i=1}^{n} \stsp$ and we use $x_u$ to elegantly denote a generic $n$-tuple $(x_{t_1}, \dots, x_{t_n})$ in $\stsp_u$.
For the empty sequence $\emptytseq$, we have that $x_{\emptytseq}$ is equal to the empty tuple~$\emptytup$ and that $\stsp_{\emptytseq}=\{ \emptytup \}$.

We will sometimes need to concatenate two increasing sequences of finite time points, for instance $u$ and $v$ in $\setoftseq{}$.
Since $u$ and $v$ can be identified with sets, we let $u \cup v$ denote their concatenation, taken to be their ordered union.
Finally, for any sequence $u = t_0, \dots, t_n$ in $\setoftseq{}$, we let $\max u \coloneqq \max \set{ t_i \colon i \in \set{1, \dots, n} }$, which, due to our convention that $u$ is increasing, is equal to $t_n$.
If $u$ is the empty sequence, then statements like ``$\max u < \cdot$'' are taken to be vacuously true.

	\subsection{Indicators, Operators and Norms}
Consider any non-empty finite set $S$, and collect all real-valued functions on $S$ in $\setoffn\gr{S}$.
An often-used type of function in $\setoffn\gr{S}$ is the \emph{indicator} of some subset $A \subseteq S$, denoted by $\indica{A}$ and defined by $\indica{A}\gr{x} \coloneqq 1$ if $x$ is an element of $S$ and $\indica{A}\gr{x} \coloneqq 0$ otherwise.
In order not to obfuscate the notation too much, for all $x$ in $S$, we write $\indica{x}$ instead of $\indica{\set{x}}$.

We now turn to operators on $\setoffn\gr{S}$.
Let $M$ be an operator with domain $\setoffn\gr{S}$ and range $\setoffn\gr{S}$.
Then $M$ is \emph{non-negatively homogeneous} if, for all $f$ in $\setoffn\gr{S}$ and all $\lambda$ in $\nnegreals{}$, $M \gr{\lambda f} = \lambda M f$.
Such operators play an important role in (imprecise) CTMCs.
Examples of non-negatively homogeneous operators that we have seen so far are $I$, $\trm$, $\tm_t$, $\lultro$ and $\lulto_t$.
If furthermore $M \gr{f + g} = M f + M g$ for all $f, g$ in $\setoffn\gr{S}$, then $M$ is \emph{linear}.
It is well-known that linear operators can be represented by matrices; previously encountered examples are $I$, $\trm$ and $\tm_t$.
% We collect all transition rate matrices that map $\setoffn\gr{S}$ to $\setoffn\gr{S}$ in $\setofalltrm\gr{S}$.

We bestow $\setoffn\gr{S}$ with the maximum norm:
\[
	\norm{f}
	\coloneqq \max \set{\abs{f\gr{x}} \colon x \in S}
	\text{ for all $f$ in $\setoffn\gr{S}$.}
\]
The maximum norm on $\setoffn\gr{S}$ induces an operator norm for non-negatively homogeneous operators:
\[
	\norm{M}
	\coloneqq \sup \set{ \norm{M f} \colon f \in \setoffn\gr{S}, \norm{f} = 1 }.
\]
% It is well-known that if $A$ is a linear operator, then
% \[
% 	\norm{A}
% 	= \max \set*{ \sum_{y \in \stsp} \abs{A\gr{x,y}} \colon x \in \stsp }.
% \]
Finally, we turn to transition rate matrices, i.e., matrices with non-negative off-diagonal elements and rows that sum to zero.
We use $\setofalltrm\gr{S}$ to denote the set of all transition rate matrices that map $\setoffn\gr{S}$ to $\setoffn\gr{S}$.
It is well-known that, for all $\trm$ in $\setofalltrm\gr{S}$,
\begin{equation}
\label{eqn:Norm of trm}
	\norm{\trm}
	= 2 \max \set*{ \abs{\trm\gr{s, s}} \colon s \in S }
	= 2 \max \set*{ -[\trm \indica{s}]\gr{s} \colon s \in S }.
\end{equation}

	\section{Extra Material for Sect.~\ref{sec:CTMCs}}
Recall from Sect.~\ref{ssec:Homogeneous CTMCs} that we can consider stochastic processes in both the classical measure-theoretic framework and the full conditional framework.
For the former, we refer to \cite[Sect.~1.1]{1991Anderson} and references therein.
Since the latter is the approach that is introduced and followed by Krak et.~al.~\cite{2017KrakDeBock}, it will be the approach that we will follow here.
Therefore, we here briefly recall the notation, terminology and results from \cite[Sects.~4 and 5]{2017KrakDeBock} that we need in the remainder: we discuss coherent conditional probabilities in Sect.~\ref{sapp:Coherent conditional probabilities}, explain how stochastic processes are coherent conditional probabilities with a specific domain in Sect.~\ref{sapp:Stochastic processes} and treat the special case of homogeneous CTMCs in Sect.\ref{sapp:Precise CTMCs}.

	\subsection{Coherent Conditional Probabilities}
\label{sapp:Coherent conditional probabilities}
Fix some non-empty set $S$ called the \emph{outcome space}.
For this outcome space $S$, we let $\ccpevents{}\gr{S}$ denote the set of all subsets of $S$, and furthermore let $\ccpevents{}\gr{S}_{\emptyset} \coloneqq \elementaryevents{}\gr{S} \setminus \emptyset$.
The following definition is one of the most elementary and essential ones that we will need throughout the remainder.
\begin{definition}
\label{def:Coherent conditional probability}
	Let $S$ be a non-empty set and $\prob$ a real-valued map from $\ccpdomain \subseteq \ccpevents\gr{S} \times \ccpevents\gr{S}_{\emptyset}$ to $\reals$.
	Then $\probm$ is a \emph{coherent conditional probability} if, for all $n$ in $\nats$, $\gr{A_1, C_1}$, $\dots$, $\gr{A_n, C_n}$ in $\ccpdomain$ and $\lambda_1, \dots, \lambda_n$ in $\reals$,
	\[
		\max \set*{ \sum_{i=1}^{n}\lambda_i \indica{C_i}\gr{s} \gr*{\prob\gr{A_i \cond C_i} - \indica{A_i}\gr{s}} \colon s \in \cup_{i=1}^{n} C_i}
		\geq 0.
	\]
\end{definition}
\begin{lemma}[Theorem~4 in \cite{1985Regazzini}]
\label{lem:Coherent contional probability can be extended}
	Let $S$ be a non-empty set.
	If $\prob$ is a coherent conditional probability on $\ccpdomain \subseteq \ccpevents\gr{S} \times \ccpevents\gr{S}_{\emptyset}$, then for any $\ccpdomain^{\star}$ such that $\ccpdomain{} \subseteq \ccpdomain^{\star} \subseteq \ccpevents\gr{S} \times \ccpevents\gr{S}_{\emptyset}$, $\prob$ can be extended to a coherent conditional probability $\prob^{\star}$ on $\ccpdomain^{\star}$, in the sense that $\prob^{*}\gr{A \cond C} = \prob\gr{A \cond C}$ for all $\gr{A, C} \in \ccpdomain{}$.
\end{lemma}
\begin{lemma}[Corollary~4.3 in \cite{2017KrakDeBock}]
\label{lem:Coherent conditional probability if and only if it can be extended}
	Let $S$ be a non-empty set.
	Then $\prob$ is a coherent conditional probability on $\ccpdomain \subseteq \ccpevents\gr{S} \times \ccpevents\gr{S}_{\emptyset}$ if and only if it can be extended to a coherent conditional probability on $\ccpevents\gr{S} \times \ccpevents\gr{S}_{\emptyset}$.
\end{lemma}
Definition~\ref{def:Coherent conditional probability} might seem rather abstract on first encounter, but our motivation for using it is the following result.
\begin{lemma}[(5)--(8) in \cite{1985Regazzini}]
\label{lem:Coherent conditional probability satisfies the laws of probability}
	Let $S$ be a non-empty set.
	If $\probm$ is a coherent conditional probability on $\ccpdomain \subseteq \ccpevents\gr{S} \times \ccpevents\gr{S}_{\emptyset}$, then for all $(A, C)$, $\gr{B, C}$ and $\gr{D, C}$ in $\ccpdomain{}$ such that $\gr{A, D \cap C}$ is in $\ccpdomain$,
	\begin{enumerate}[label=\upshape{(\roman*)}, leftmargin=*]
		\item $\probm\gr{A \cond C} \geq 0$;
		\item $\probm\gr{A \cond C} = 1$ if $C \subseteq A$;
		\item $\probm\gr{A \cup B \cond C} = \probm\gr{A \cond C} + \probm\gr{B \cond C}$ if $A \cap B = \emptyset$;
		\item \label{LOP:Bayes rule} $\probm\gr{A \cap D \cond C} = \probm\gr{A \cond D \cap C} \probm\gr{D \cond C}$.
	\end{enumerate}
\end{lemma}
Lemma~\ref{lem:Coherent conditional probability satisfies the laws of probability} states that a coherent conditional probability satisfies the standard laws of (conditional) probability on its domain: properties (i)--(iii) state that $\prob\gr{\cdot \cond C}$ is a probability measure, while (iv) is Bayes' rule.

	\subsection{Stochastic Processes}
\label{sapp:Stochastic processes}
Fix some finite state space $\stsp$.
We are then uncertain about what the actual path $\pad \colon \nnegreals \to \stsp$ of the system will be.
We therefore consider a set of paths $\paths$, which contains all ``feasible'' paths.
The only thing that is required of $\paths$ is that
\begin{equation}
\label{eqn:Set of feasible paths}
	\gr{\forall u \in \setoftseqne}
	\gr{\forall x_u \in \stsp_u}
	\gr{\exists \omega \in \paths}
	~\omega\gr{t}
	= x_t
	~\text{for all $t$ in $u$.}
\end{equation}
% for all $u=t_1, \dots, t_n$ in $\setoftseqne$ and $x_u = \gr{x_1, \dots, x_n}$ in $\stsp_u$, there is at least one $\omega$ in $\paths$ such that, for all $i$ in $\set{1, \dots, n}$, $\omega\gr{t_i} = x_i$.
For all $t$ in $\nnegreals$ and $x$ in $\stsp$, we then define the basic event
\[
	\gr{X_t = x}
	\coloneqq \set{\pad \in \paths \colon \pad\gr{t} = x}.
\]
Similarly, for all $u$ in $\setoftseq$ and $x_u$ in $\stsp_u$, we let
\[
	\gr{X_u = x_u}
	\coloneqq \bigcap_{t \in u} \gr{X_t = x_t}.
\]
We follow the convention that an empty intersection in expressions similar to the one above correspond to $\paths$; hence $\gr{X_u = x_u} = \paths$ if $u$ is the empty sequence $\emptytseq$.

For all $u$ in $\setoftseq$, the set of elementary events
\[
	\elementaryevents_u
	\coloneqq \set{\gr{X_t = x} \colon t \in u \cup [\max u, +\infty), x \in \stsp}
\]
induces an algebra $\eventalgebra_u \coloneqq \langle \elementaryevents_u \rangle$.
We use these algebras to define the domain
\[
	\procdomain
	\coloneqq \set{\gr{A_u, X_u = x_u} \colon u \in \setoftseq, x_u \in \stsp_u, A_u \in \eventalgebra_u},
\]
and consider maps of the form
\[
	\prob \colon \procdomain \to \reals \colon (A_u, X_u = x_u) \mapsto \prob\gr{A_u \cond X_u = x_u},
\]
where---in order to not to unnecessarily clutter our notation---we leave out the conditioning event if it is $\gr{X_\emptytseq = x_\emptytseq} = \paths$:
\[
	\probm\gr{A}
	\coloneqq \probm\gr{A \cond X_{\emptyset} = x_{\emptyset}}
	= \probm\gr{A \cond \paths}
	\quad\text{for any $A$ in $\eventalgebra_\emptytseq$.}
\]
\begin{definition}[Definition~4.3 in \cite{2017KrakDeBock}]
\label{def:Stochastic process}
	A real-valued map $\prob$ on $\procdomain$ is a \emph{stochastic process} if it is a coherent conditional probability on $\procdomain$.
\end{definition}
It immediately follows from Lemma~\ref{lem:Coherent conditional probability satisfies the laws of probability} that a stochastic process $\prob$ satisfies the laws of (conditional) probability.
Because these laws are so well-known, we will frequently use them without explicitly referring to Lemma~\ref{lem:Coherent conditional probability satisfies the laws of probability}.

	\subsection{Precise (Homogeneous) Continuous-Time Markov Chains As Special Cases}
\label{sapp:Precise CTMCs}
The following is a more formal definition of the terms introduced in Sect.~\ref{ssec:Homogeneous CTMCs}.
\begin{definition}
\label{def:HCTMC as stochastic process}
	A stochastic process~$\prob \colon \procdomain \to \reals$ is a \emph{continuous-time Markov chain} (CTMC) if, for all $t, \Delta$ in $\nnegreals$, $u$ in $\setoftseql{t}$, $x, y$ in $\stsp$ and $x_u$ in $\stsp_u$,
	\[
		\probm\gr{ X_{t+\Delta}=y \cond X_u = x_u, X_t = x }
		= \probm\gr{ X_{t+\Delta}=y \cond X_t = x }.
	\]
	This CTMC~$\probm$ is \emph{homogeneous} if furthermore
	\[
		\probm\gr{ X_{t+\Delta}=y \cond X_t = x }
		= \probm\gr{ X_{\Delta}=y \cond X_0 = x }
	\]
	for all $t, \Delta$ in $\nnegreals$ and $x, y$ in $\stsp$.
\end{definition}

Our statement in Sect.~\ref{ssec:Homogeneous CTMCs} that a homogeneous CTMC is uniquely characterised by the triplet $\gr{\stsp, \inidist, \trm}$ is justified due to the following result.
\begin{proposition}[Theorem~5.2 in \cite{2017KrakDeBock}]
\label{prop:Unique homogeneous CTMC}
	Let $\stsp$ be a state space, $\inidist$ a distribution on $\stsp$ and $\trm$ a transition rate matrix.
	Then there is a unique homogeneous CTCM~$P$ such that (i) $\probm\gr{X_0 = x} = \inidist\gr{x}$ for all $x$ in $\stsp$ and (ii)~$\probm\gr{ X_{t+\Delta} = y \cond X_t = x} = T_{\Delta}\gr{x, y}$ for all $x, y$ in $\stsp$ and $t, \Delta$ in $\nnegreals$.
\end{proposition}

In this appendix, a \emph{positive and irreducible CTMC} is any stochastic process~$P$ that is a homogeneous CTMC and that is (uniquely) characterised by a positive initial distribution~$\inidist$ and an irreducible transition rate matrix~$\trm$.

	\subsection{Irreducibility}
\label{sapp:Irreducibility}
% For completeness' sake, we here repeat the definition of irreducibility given in Section~\ref{ssec:Irreducibility of homogeneous CTMCs}.
% \begin{definition}
% \label{def:Irreducibility}
% 	The transition rate matrix $\trm$ is called \emph{irreducible} if, for all $t$ in $\posreals$ and $x, y$ in $\stsp$, $\tm_t \gr{x, y} > 0$.
% \end{definition}
An easy to check necessary and sufficient condition for irreducibility is based on the accessibility relation $\cdot \acces \cdot$ \cite{1997Norris}.
We say that a state $x$ is \emph{accessible} from a state $y$ (or that $y$ \emph{leads} to $x$) if there is a sequence $y = x_{0}, x_{1} \dots, x_{n} = x$ in $\stsp$ such that $\trm\gr{x_{i-1}, x_{i}} > 0$ for all $i$ in $\set{1, \dots, n}$.
\begin{proposition}[Theorem~3.2.1 in \cite{1997Norris}]
\label{prop:Precise irreducibility}
	The transition rate matrix $\trm$ is irreducible if and only if every state is accessible from any other state.
	More formally, this condition reads
	\[
		\stsp_{\mathrm{top}}
		\coloneqq \set*{ x \in \stsp \colon (\forall y \in \stsp) y \acces x }
		= \stsp.
	\]
\end{proposition}

The following lemma is our main reason for assuming that the CTMC has a positive initial distribution and an irreducible transition rate matrix.
\begin{lemma}
\label{lem:Irreducible CTMC:Any path has pos prob}
	If $\markov$ is a positive and irreducible CTMC, then for any $u = t_{1}, \dots, t_{n}$ in $\setoftseqne$ and $x_{u} = (x_{1}, \dots, x_{n})$ in $\stsp_{u}$,
	\[
		\probm\gr{X_{u} = x_{u}}
		= \probm\gr{X_{t_{1}} = x_{1}, \dots, X_{t_{n}} = x_{n}}
		> 0.
	\]
\end{lemma}
\begin{proof}
	Repeated application of Bayes' rule and the Markov property \eqref{eqn:HCTMC Markov} yields
	\begin{multline*}
		\probm\gr{ X_{u} = x_{u} }
		= \probm\gr{X_{t_{1}} = x_{1}, \dots, X_{t_{n}} = x_{n}} \\
		= \sum_{x_{0} \in \stsp}
		\probm\gr{X_{0} = x_{0}}
		\probm\gr{X_{t_1} = x_{1} \cond X_{0} = x_{0}} \prod_{i=2}^{n} \probm\gr{X_{t_i} = x_{i} \cond X_{t_{i-1}} = x_{i-1}}.
	\end{multline*}
	We now use \eqref{eqn:HCTMC initial} and \eqref{eqn:HCTMC transition} to obtain
	\begin{equation}
	\label{eqn:Probability of state assignment as sum of products}
		\probm\gr{ X_{u} = x_{u} }
		= \sum_{x_{0} \in \stsp}
		\inidist\gr{x_{0}}
		\tm_{t_{1}}\gr{x_0, x_1} \prod_{i=2}^{n} \tm_{\gr{t_i - t_{i-1}}}\gr{x_{i-1}, x_{i}}.
	\end{equation}
	% Recall from Definition~\ref{def:Irreducibility}
	As $\trm$ is irreducible, $\tm_{t}\gr{x, y}$ is positive for all $t$ in $\posreals$ and all $x, y$ in $\stsp$.
	Hence, all terms in the product on the right hand side of \eqref{eqn:Probability of state assignment as sum of products} are positive.
	We now distinguish two cases: $t_{1} > 0$ and $t_{1} = 0$.
	In the first case, it again follows from the irreducibility that all $\tm_{t_{1}}\gr{x_0, x_1}$ are positive.
	In the second case, $\tm_{t_{1}}\gr{x_0, x_1}$ is zero if $x_0 \neq x_1$ and $1$ if $x_0 = x_1$.
	Furthermore, $\inidist\gr{x_0} > 0$ for all $x_0$ in $\stsp$ by assumption.
	The stated now follows by observing that at least one of the terms in the sum is a product of positive real numbers and therefore positive itself, and that the other terms are non-negative.
	\qed
\end{proof}

We will also need the following properties.
The first one is essentially well-known, but we could not immediately find a good reference for it.
\begin{lemma}
\label{lem:I + delta trm is tm}
	Let $\trm$ be a transition rate matrix and $\delta$ in $\posreals$ such that $\delta \norm{\trm} < 2$.
	Then $\gr{I + \delta \trm}$ is a transition matrix.
	If $\trm$ is furthermore irreducible, then $\gr{I + \delta \trm}$ is aperiodic and irreducible in the sense of \cite{1997Norris}.
\end{lemma}
\begin{proof}
	Fix any $\delta$ in $\posreals$ such that $\delta \norm{\trm} < 2$.
	It can then be immediately verified---see for instance \cite[p.~289]{2009Stewart}---that $\tm$ is a transition matrix.

	That $\tm$ is irreducible follows from the irreducibility of $\trm$.
	Recall from Proposition~\ref{prop:Precise irreducibility} that the irreducibility of $\trm$ implies that for any $x, y$ in $\stsp$ such that $x \neq y$, there is a sequence $y=x_0, \dots, x_n$ in $\stsp$ such that $\trm\gr{x_{i-1}, x_i} > 0$ for all $i$ in $\set{1, \dots, n}$.
	Clearly, this implies that, for all $i$ in $\set{1, \dots, n}$,
	\[
		\tm\gr{x_{i-1}, x_i} = \gr{I + \delta \trm}\gr{x_{i-1}, x_i} = I\gr{x_{i-1}, x_i} + \delta \trm\gr{x_{i-1}, x_i} > 0.
	\]
	From \cite[Theorem~1.2.1]{1997Norris}, it follows that $y$ leads to $x$ with respect to $\tm$.
	Since any two distinct states are communicating with respect to $\tm$, we conclude that $\tm$ is irreducible.

	That $\tm$ is aperiodic follows from \cite[p.~304]{2009Stewart}.
	\qed
\end{proof}
\begin{lemma}
\label{lem:Lower and upper bound for limit expectation}
	If $\trm$ is an irreducible transition rate matrix, then for all $f$ in $\setoffn\gr{\stsp}$, $\delta$ in $\posreals$ such that $\delta \norm{\trm} < 2$ and $n$ in $\natz{}$,
	\[
		\min \gr{I + \delta \trm}^n f
		\leq \limdist{} f,
		% \leq \max \gr{I + \delta \trm}^n f,
	\]
	where $\limdist{}$ is the stationary distribution of $\trm$.
\end{lemma}
\begin{proof}
	Fix any $\delta$ in $\posreals$ such that $\delta \norm{\trm} < 2$, and let $\tm \coloneqq I + \delta \trm$.
	We recall from Lemma~\ref{lem:I + delta trm is tm} that $\tm$ is an aperiodic and irreducible transition matrix.
	Furthermore, from the equilibrium condition $\limdist \trm = 0$ it follows that
	\[
		\limdist \tm
		= \limdist \gr{I + \delta \trm}
		= \limdist + \delta \limdist \trm = \limdist.
	\]
	Since $\limdist$ is an invariant distribution for the aperiodic and irreducible transition matrix $\tm$, it follows from \cite[Theorem~1.8.3]{1997Norris} that $\lim_{n \to +\infty} [\tm^n f](x) = \limdist f$ for all $f$ in $\setoffn\gr{\stsp}$ and $x$ in $\stsp$.

	Fix now any $f$ in $\setoffn\gr{\stsp}$ and consider the sequence
	\[
		\{ \min \gr{I + \delta \trm}^n f  \}_{n \in \natz}
		= \{ \min \tm^n f  \}_{n \in \natz}.
	\]
	From the previous, we know that this sequence converges to $\limdist f$ in the limit for $n \to +\infty$.
	Since $\tm$ is a transition matrix (a matrix with non-negative elements and rows that sum to $1$), it clearly holds that $\min g \leq \min \tm g$ for all $g$ in $\setoffn\gr{\stsp}$.
	It now follows from repeated application of this inequality that the sequence $\{ \min \tm^n f  \}_{n \in \natz}$ is non-decreasing, which proves the stated.
	\qed
\end{proof}

	\section{Imprecise Continuous-Time Markov Chains: A Brief Summary}
\label{app:iCTMCs}
In this supplementary section, we briefly introduce the notation, terminology and results concerning imprecise CTMCs \cite{2017KrakDeBock,2016DeBock,2015Skulj} that we will need in the remainder.

	\subsection{Sets of Consistent Processes and Lower Expectations}
\label{sapp:Sets of consistent processes and lower expectations}
In general, the main idea behind imprecise CTMCs is to consider a set of stochastic process instead of a single stochastic process.
In particular, Krak et.~al.~\cite{2017KrakDeBock} focus on three nested sets of processes, all characterised by a non-empty set of initial distributions $\credset$ and a non-empty bounded set of transition rate matrices~$\setoftrm{} \subseteq \setofalltrm\gr{\stsp}$.
More specifically, they collect in $\setofconsproc_{\setoftrm, \credset}^{\mathrm{W}}$ all stochastic processes that are: (i) well-behaved, a technical condition \cite[Definition~4.4]{2017KrakDeBock}; (ii) consistent with $\setoftrm$, in the sense that at all times the ``instantaneous transition rate matrix'' is contained in $\setoftrm$ \cite[Definition~6.1]{2017KrakDeBock}; and (iii) consistent with $\credset$, in the sense that $\credset$ contains the initial distribution \cite[Definition~6.2]{2017KrakDeBock}.
Similarly, $\setofconsproc_{\setoftrm, \credset}^{\mathrm{WM}}$ ($\setofconsproc_{\setoftrm, \credset}^{\mathrm{WHM}}$) contains all well-behaved (homogeneous) Markov processes that are consistent with $\setoftrm$ and $\credset$.
These sets are clearly nested, in the sense that
\begin{equation}
\label{eqn:Nested sets of consistent processes}
	\setofconsproc_{\setoftrm, \credset}^{\mathrm{WHM}}
	\subseteq \setofconsproc_{\setoftrm, \credset}^{\mathrm{WM}}
	\subseteq \setofconsproc_{\setoftrm, \credset}^{\mathrm{W}}.
\end{equation}

Using these sets of consistent  stochastic processes, Krak et.~al.~\cite{2017KrakDeBock} construct lower (and conjugate upper) expectations as follows.
For any non-empty set of initial distributions $\credset$ and non-empty bounded set of transition rate matrices~$\setoftrm$, they let
\[
	\underline{\prev}_{\setoftrm, \credset}^{\mathrm{W}}\gr{ \cdot \cond \cdot}
	\coloneqq \inf \set{ \prev_\prob\gr{ \cdot \cond \cdot} \colon \prob \in \setofconsproc_{\setoftrm, \credset}^{\mathrm{W}} },
\]
where $\prev_\prob$ denotes the expectation with respect to the process $\prob$, and similarly for $\underline{\prev}_{\setoftrm, \credset}^{\mathrm{WM}}$ and $\underline{\prev}_{\setoftrm, \credset}^{\mathrm{WHM}}$.
It is now intuitively clear from \eqref{eqn:Nested sets of consistent processes} that
\[
	\underline{\prev}_{\setoftrm, \credset}^{\mathrm{W}}\gr{ \cdot \cond \cdot}
	\leq \underline{\prev}_{\setoftrm, \credset}^{\mathrm{WM}}\gr{ \cdot \cond \cdot}
	\leq \underline{\prev}_{\setoftrm, \credset}^{\mathrm{WHM}}\gr{ \cdot \cond \cdot}.
\]

	\subsection{Lower Transition (Rate) Operators}
\label{sapp:LTROs}
With any non-empty bounded set of transition rate matrices $\setoftrm{}$, we associate the operator $\ltro \colon \setoffn\gr{\stsp} \to \setoffn\gr{\stsp} \colon f \mapsto \ltro f$ where, for all $f$ in $\setoffn\gr{\stsp}$,
\begin{equation}
\label{eqn:ltro}
	[\ltro f]\gr{x}
	\coloneqq \inf \set*{ [\trm f]\gr{x} \colon \trm \in \setoftrm }
	\text{ for all $x$ in $\stsp$.}
\end{equation}
This operator~$\ltro$ is called the \emph{lower envelope} of $\setoftrm$.
By \cite[Proposition~7.5]{2017KrakDeBock}, we know that it is a \emph{lower transition rate operator} \cite[Definition~7.2]{2017KrakDeBock}, a specific type of non-homogeneous operator that is a non-linear generalisation of the concept of a transition rate matrix.
Hence, it should not come as a surprise that there is an equivalent of \eqref{eqn:Matrix exponential}.
Indeed, for any $t$ in $\nnegreals$, one defines the \emph{lower transition operator over $t$} as
\begin{equation}
\label{eqn:lto}
	\lto_t
	\coloneqq \lim_{n \to +\infty} \gr*{I + \frac{t}{n} \ltro}^n,
\end{equation}
where the $n$-th power should be interpreted as $n$ consecutive applications.

Almost everything has now been set up to state Proposition~\ref{prop:Imprecise Markov property}, which is the main result from imprecise CTMCs that we will need; we just have to introduce the following definition.
\begin{definition}[Definition~7.3 in \cite{2017KrakDeBock}]
\label{def:separately specified rows}
	A non-empty set of transition rate matrices $\setoftrm \subseteq \setofalltrm\gr{\stsp}$ has \emph{separately specified rows} if for any $\card{\stsp}$-tuple $\gr{\trm_x}_{x \in \stsp}$ with entries that are all elements of $\setoftrm$, there is a $\trm^\star$ in $\setoftrm$ such that
	\[
		\trm^\star\gr{x, y}
		= \trm_x\gr{x, y}
		\text{ for all } x, y \in \stsp.
	\]
	% \[
	% 	\setoftrm{}
	% 	= \set*{ \trm \in \setofalltrm\gr{\stsp} \colon \gr{\forall x \in \stsp} \trm\gr{x, \cdot} \in \setoftrm_x },
	% \]
	% where, for all $x$ in $\stsp$,
	% \[
	% 	\setoftrm_x
	% 	\coloneqq \set*{ \trm'\gr{x, \cdot} \colon \trm' \in \setoftrm }
	% \]
	% is the set of rows from which the $x$-row $\trm\gr{x, \cdot}$ of $\trm$ is selected, independent of the other rows.
\end{definition}
\begin{proposition}[Corollary~8.3 in \cite{2017KrakDeBock}]
\label{prop:Imprecise Markov property}
	Let $\credset$ be a non-empty set of initial distributions and $\setoftrm$ a non-empty and bounded set of transition rate matrices that has separately specified rows.
	If $\ltro$ is the corresponding lower transition rate operator \eqref{eqn:ltro}, then for any $t, \Delta$ in $\nnegreals$, $u$ in $\setoftseql{t}$, $x$ in $\stsp$, $x_u$ in $\stsp_u$ and $f$ in $\setoffn\gr{\stsp}$,
	\begin{multline}
	\label{eqn:Imprecise Markov property}
		\underline{\prev}_{\setoftrm, \credset}^{\mathrm{W}}\gr{ f\gr{X_{t+\Delta}} \cond X_u = x_u, X_t = x}
		= \underline{\prev}_{\setoftrm, \credset}^{\mathrm{WM}}\gr{ f\gr{X_{t+\Delta}} \cond X_u = x_u, X_t = x} \\
		= [\lto_{\Delta} f](x).
	\end{multline}
	% where $\lto_{\Delta}$ is the corresponding lower transition operator over $\Delta$ \eqref{eqn:lto}.
\end{proposition}
This result justifies calling $\setofconsproc_{\setoftrm, \credset}^{\mathrm{W}}$ (and $\setofconsproc_{\setoftrm, \credset}^{\mathrm{WM}}$) an imprecise CTMC, as \eqref{eqn:Imprecise Markov property} is an imprecise version of the Markov property \eqref{eqn:HCTMC Markov}.
Even more, the imprecise CTMC~$\setofconsproc_{\setoftrm, \credset}^{\mathrm{W}}$ also satisfies an imprecise version of the law of iterated expectation.
\begin{proposition}[Theorem~6.5 in \cite{2017KrakDeBock}]
\label{prop:Imprecise law of iterated expectation}
	If $\credset$ is a non-empty set of initial distributions and $\setoftrm$ a non-empty, bounded and convex set of transition rate matrices, then for any $u, v, w$ in $\setoftseq$ with $\max u < \min v$ and $\max v < \min w$, $x_u$ in $\stsp_u$ and $f$ in $\setoffn\gr{\stsp_{u \cup v \cup w}}$,
	\begin{multline*}
		\underline{\prev}_{\setoftrm, \credset}^{\mathrm{W}}\gr{f\gr{X_u, X_v, X_w} \cond X_u = x_u} \\
		= \underline{\prev}_{\setoftrm, \credset}^{\mathrm{W}}\gr{ \underline{\prev}_{\setoftrm, \credset}^{\mathrm{W}}\gr{f\gr{X_u, X_v, X_w} \cond X_u = x_u, X_v} \cond X_u = x_u}.
	\end{multline*}
\end{proposition}

We conclude this section with a strengthened version of \cite[LR5]{2017KrakDeBock}.
\begin{lemma}
\label{lem:Norm of ltro}
	Let $\setoftrm$ be a non-empty and bounded set of transition rate matrices with associated lower transition rate operator $\ltro$.
	Then $\norm{\ltro} = \sup \set*{\norm{\trm} \colon \trm \in \setoftrm}$, such that $\norm{\trm} \leq \norm{\ltro}$ for all $\trm$ in $\setoftrm$.
\end{lemma}
\begin{proof}
	Recall from \eqref{eqn:Norm of trm} that, for all $\trm$ in $\setoftrm$,
	\[
		\norm{\trm}
		= 2 \max \set*{ \abs{\trm\gr{x, x}} \colon x \in \stsp }
		= 2 \max \set*{ -[\trm \indica{x}]\gr{x} \colon x \in \stsp }.
	\]
	Moreover, by \cite[Proposition~4]{2017Erreygers},
	\begin{align*}
		\norm{\ltro}
		% = 2 \max \set*{ \abs{[\ltro \indica{x}]\gr{x}} \colon x \in \stsp }
		&= 2 \max \set*{ -[\ltro \indica{x}]\gr{x} \colon x \in \stsp }.
	\intertext{
	Using \eqref{eqn:ltro} and executing some straightforward manipulations yields
	}
		\norm{\ltro}
		&= 2 \max \set*{ - \inf \set*{ [\trm \indica{x}]\gr{x} \colon \trm \in \setoftrm } \colon x \in \stsp } \\
		&= 2 \max \set*{ \sup \set*{ - [\trm \indica{x}]\gr{x} \colon \trm \in \setoftrm } \colon x \in \stsp } \\
		&= \sup \set*{ 2 \max \set*{ - [\trm \indica{x}]\gr{x} \colon x \in \stsp } \colon \trm \in \setoftrm } = \sup \set*{ \norm{\trm} \colon \trm \in \setoftrm}.
	\end{align*}
	The stated now follows immediately from the final equality.
	\qed
\end{proof}

	\subsection{Irreducibility}
Just like precise CTMCs, their imprecise counterparts also have some nice ergodic properties.
For a detailed exposition of these properties, we refer the interested reader to our previous work \cite{2016DeBock,2017Erreygers}.
We here only mention the definitions and results that we will need in the remainder.

Let $\setoftrm$ be a non-empty and bounded set of lower transition rate operators.
As previously mentioned in Appendix~\ref{sapp:LTROs}, the corresponding lower envelope $\ltro$ is a lower transition rate operator.
For such a lower transition rate operator, the imprecise counterpart of the accessibility relation $\cdot \acces \cdot$ is the upper reachability relation \cite[Definition~7]{2016DeBock}.
We say that a state $x$ is \emph{upper reachable} from the state $y$, denoted by $y \upacces x$, if there is a sequence $y = x_0, \dots, x_n = x$ in $\stsp$ such that $- [\ltro \gr{ - \indica{x_i}}]\gr{x_{i-1}} > 0$ for all $i$ in $\set{1, \dots, n}$.
\begin{definition}
\label{def:Imprecise irreducibility}
	Let $\setoftrm$ be an non-empty bounded set of transition rate matrices.
	The corresponding lower transition rate operator $\ltro$ is \emph{irreducible} if any state is upper reachable from any other state, that is, if
	\[
		\stsp_{\mathrm{top}}
		\coloneqq \set*{x \in \stsp \colon \gr{\forall y \in \stsp} y \upacces x}
		= \stsp.
	\]
\end{definition}

It now follows from \cite[Theorem~19]{2016DeBock} that if $\ltro$ is irreducible, then $\ltro$ is \emph{ergodic}, meaning that, for all $f$ in $\setoffn\gr{\stsp}$, $\lto_t f$ converges to a constant function in the limit for $t \to +\infty$ \cite[Definition~6]{2016DeBock}.
For all $t$ in $\posreals$ and $x,y$ in $\stsp$, this also implies that $- [\lto_t \gr{- \indica{x}}]\gr{y} > 0$ \cite[Proposition~17]{2016DeBock}, which is similar to
% Definition~\ref{def:Irreducibility}
the definition of irreducibility in the precise case.
Note also the similarity between Proposition~\ref{prop:Precise irreducibility} and Definition~\ref{def:Imprecise irreducibility}, which justifies the use of the term irreducible.
Furthermore, the following property holds.
\begin{corollary}
\label{cor:lto skeleton of irreducible ltro converges}
	Let $\setoftrm$ be a non-empty bounded set of transition rate matrices.
	If the corresponding lower envelope $\ltro$ is irreducible, then for any $f$ in $\setoffn\gr{\stsp}$ and $\delta$ in $\posreals$ such that $\delta \norm{\ltro} < 2$, $\gr{I + \delta \ltro}^n f$ converges to a constant function in the limit for $n \to +\infty$: there is some $f_{\delta}$ in $\reals$ such that \smash{$\lim_{n \to +\infty} \gr{I + \delta \ltro}^n f = f_{\delta} \indica{\stsp}$}.
	Moreover, $\set{\min \gr{I + \delta \ltro}^n f}_{n \in \nats}$ is a non-decreasing sequence that converges to $f_{\delta}$.
\end{corollary}
\begin{proof}
	Fix any $\delta$ in $\posreals$ such that $\delta \norm{\ltro} < 2$ and let $\lto \coloneqq \gr{I + \delta \ltro}$.
	Then by \cite[Proposition~3]{2017Erreygers}, $\lto$ is a \emph{lower transition operator} (see \cite[Definition~7.1]{2017KrakDeBock} or \cite[Definition~1]{2017Erreygers}).
	Furthermore, since $\ltro$ is irreducible and hence ergodic, it follows from \cite[Theorem~8]{2017Erreygers} and either \cite[Proposition~7]{2012Hermans} or \cite[Theorem~21]{2013Skulj} that the lower transition operator $\lto$ is also \emph{ergodic}, meaning that, for all $f$ in $\setoffn\gr{\stsp}$, $\lim_{n \to +\infty} \lto^n f = \lim_{n \to +\infty} \gr{I + \delta \ltro}^n f$ exists and is a constant function, here denoted by $f_{\delta} \indica{\stsp}$.
	Finally, the non-decreasing character of the sequence in the statement can be verified by repeatedly applying \cite[Definition~7.1(LT1)]{2017KrakDeBock}; that the sequence converges to $f_{\delta}$ follows immediately from the previous.
	\qed
\end{proof}

	\section{The Lumped Stochastic Process}
Before diving in head first, we first extend the inverse lumping map $\invlumap$ to tuples of state assignments.
For any $u$ in $\setoftseq$, similar to what we did in Section~\ref{sapp:Sequences of time points}, we let $\lu{x}_u$ denote an element of $\smash{\lustsp_u \coloneqq \prod_{t \in u} \lustsp}$.
The domain of the inverse lumping map $\invlumap$ can then be trivially extended to $\lustsp_u$ as follows: we let
$
	\invlumap\gr{\lu{x}_\emptytseq}
	\coloneqq x_{\emptytseq}
$
and for all $u$ in $\setoftseqne{}$ and all $\lu{x}_u$ in $\lustsp_u$, we let
\[
	\invlumap\gr{\lu{x}_{u}}
	\coloneqq \set*{ x_{u} \in \stsp_{u} \colon (\forall t \in u) \lumap\gr{x_t} = \lu{x}_t }.
\]

	\subsection{A Formal Definition of the Lumped Stochastic Process}
\label{sapp:Formal definition of the lumped proc}
In order to define the lumped process rigorously, we need a more formal construction than that given in the main text \eqref{eqn:Naive definition of lumped process}.
To that end, we now consider a positive and irreducible CTMC $\prob$ that is characterised by $(\stsp, \inidist, \trm)$ and a lumping map $\lumap \colon \stsp \to \lustsp$.

For starters, we first need to specify the set of ``feasible'' paths for the lumped process.
A natural way is to map $\paths$, the set of ``feasible'' paths on $\stsp$, to a set of paths on $\lustsp$ using $\lumap$:
\begin{equation}
\label{eqn:Set of lumped paths}
	\lu\paths
	\coloneqq \set*{\lumap \circ \pad \colon \pad \in \paths},
\end{equation}
where $\lumap \circ \pad$ denotes the function composition of $\pad \colon \nnegreals \to \stsp$ and $\lumap \colon \stsp \to \lustsp$, given by $\lumap \circ \pad \colon \nnegreals{} \to \lustsp \colon t \mapsto \gr{\lumap \circ \pad} \gr{t} \coloneqq \lumap\gr{\pad\gr{t}}$.
Note that because $\paths$ satisfies \eqref{eqn:Set of feasible paths}, $\lu\paths$ clearly satisfies a lumped version of \eqref{eqn:Set of feasible paths}:
\[
	\gr{\forall u \in \setoftseqne}
	\gr{\forall \lu{x}_u \in \lustsp_u}
	\gr{\exists \lu\omega \in \lu\paths}
	~\lu\omega\gr{t}
	= \lu{x}_t
	~\text{for all $t$ in $u$.}
\]

For any $t$ in $\nnegreals$ and any $\lu{x}$ in $\lustsp$, we can now consider the elementary event
\[
	\gr{\lu{X}_t = \lu{x}}
	\coloneqq \set{\lu\pad \in \lu\paths \colon \lu\pad\gr{t} = \lu{x}}.
\]
% In light of \eqref{eqn:Naive definition of lumped process} and \eqref{eqn:Set of lumped paths}, $\gr{\lu{X}_t = \lu{x}}$ in $\lu\paths$ corresponds to an event in $\paths$:
% \begin{align*}
% 	\set*{\pad \in \paths \colon \lumap\gr{\pad\gr{t}} = \lu{x}}
% 	&= \set*{\pad \in \paths \colon \pad\gr{t} \in \invlumap\gr{\lu{x}}}
% 	= \bigcup_{x \in \invlumap\gr{\lu{x}}} \set*{\pad \in \paths \colon \pad\gr{t} = x} \\
% 	&= \bigcup_{x \in \invlumap\gr{\lu{x}}} \gr{X_t = x}
% 	\eqqcolon \gr{X_t \in \invlumap\gr{\lu{x}}}.
% \end{align*}
As before, for any $u$ in $\setoftseq{}$ and $\lu{x}_u$ in $\lustsp_u$, we also let
\[
	\gr{\lu{X}_u = \lu{x}_u}
	\coloneqq \bigcap_{t \in u} \gr{\lu{X}_t = \lu{x}_t},
\]
where $\gr{\lu{X}_\emptytseq = \lu{x}_\emptytseq} = \lu\paths$.
% Naturally, this also corresponds to an event in $\paths$:
% \[
% 	\gr{X_u \in \invlumap\gr{\lu{x}_u}}
% 	\coloneqq \bigcup_{x_u \in \invlumap\gr{x_u}} \gr{X_u = x_u}.
% \]
For any $u$ in $\setoftseq$, the set of elementary elements
\begin{align*}
	\lu\elementaryevents_u
	\coloneqq{}& \set{\gr{\lu{X}_t = \lu{x}} \colon t \in u \cup [\max u, +\infty), \lu{x} \in \lustsp}
\end{align*}
induces the algebra $\lu\eventalgebra_u \coloneqq \langle \lu\elementaryevents_u \rangle$.
The domain of the lumped stochastic process~$\problum$ should hence be
\[
	\luprocdomain
	\coloneqq \set{\gr{\lu{A}_u, \lu{X}_u = \lu{x}_u} \colon u \in \setoftseq, \lu{x}_u \in \lustsp_u, \lu{A}_u \in \lu\eventalgebra_u}.
\]

We have now introduced almost all concepts to formally define the lumped stochastic process~$\problum$.
The sole remaining concept that we need is another inverse derived from $\lumap$, this time from $\lu\paths$ to $\paths$.
To that end, we consider the map $\invlumap_\paths$ that maps any subset $\lu{A}$ of $\lu\paths$ to
\begin{equation}
\label{eqn:invlumap_paths}
	\invlumap_\paths\gr{\lu{A}}
	\coloneqq \set{\pad \in \paths \colon \lumap \circ \pad \in \lu{A}},
\end{equation}
which is a subset of $\paths$.
Note that $\invlumap_\paths$ is indeed an inverse, as clearly
\begin{equation}
\label{eqn:invlumap_paths is inverse}
	\set{\lumap \circ \pad \colon \pad \in \invlumap_\paths\gr{\lu{A}} }
	= \lu{A}.
\end{equation}
Fix any $u$ in $\setoftseq$ and $\lu{x}_u$ in $\lustsp_u$.
Then some straightforward manipulations---similar to those used in the proof of Lemma~\ref{lem:invlumap_paths is element of eventalgebra}---yield
\begin{equation}
\label{eqn:invlumap_paths:state instantiation}
	\invlumap_\paths\gr{\lu{X}_u = \lu{x}_u}
	= \bigcup_{x_u \in \invlumap\gr{\lu{x}_u}} \gr{X_u = x_u}
	\eqqcolon \gr{X_{u} \in \invlumap\gr{\lu{x}_u}},
\end{equation}
More generally, we find the following.
\begin{lemma}
\label{lem:invlumap_paths is element of eventalgebra}
	If $\lumap \colon \stsp \to \lustsp$ is a lumping map, then for all $u$ in $\setoftseq$ and $\lu{A}_u$ in $\lu\eventalgebra_u$, $\invlumap_\paths\gr{\lu{A}_u}$ is an element of $\eventalgebra_u$.
\end{lemma}
\begin{proof}
	First, we observe that from \eqref{eqn:Set of lumped paths} and \eqref{eqn:invlumap_paths} it follows that $\invlumap_\paths\gr{\lu{A}} = \emptyset$ if and only if $\lu{A} = \emptyset$.
	Next, we fix some $u$ in $\setoftseq$ and some $\lu{A}_u$ in $\lu\eventalgebra_u$ such that $\lu{A}_u \neq \emptyset$.
	Because $\lu\eventalgebra_u$ is an algebra generated by the elementary events in $\lu\elementaryevents_u$---see for instance also \cite[Proof of Lemma~C.3]{2017KrakDeBock}---there is some time sequence $v$ in $\setoftseq$ and a non-empty set of tuples $\lu{S} \subseteq \lustsp_{u \cup v}$ such that $\max u < \min v$ and
	\[
		\lu{A}_u
		= \bigcup_{\lu{z}_{u \cup v} \in \lu{S}} \gr{\lu{X}_v = \lu{z}_v}.
	\]
	If we let $w \coloneqq u \cup v$, then
	\begin{align*}
		\invlumap_\paths\gr{\lu{A}_u}
		&= \set*{\pad \in \paths \colon \lumap \circ \pad \in \bigcup_{\lu{z}_w \in \lu{S}} \gr{\lu{X}_w = \lu{z}_w}} \\
		&= \bigcup_{\lu{z}_w \in \lu{S}} \set{\pad \in \paths \colon \lumap \circ \pad \in \gr{\lu{X}_w = \lu{z}_w}}.
	\end{align*}
	Using the definition of $\gr{\lu{X}_w = \lu{z}_w}$ and \eqref{eqn:Set of lumped paths}, we write this as
	\begin{align*}
		\invlumap_\paths\gr{\lu{A}_u}
		&= \bigcup_{\lu{z}_w \in \lu{S}} \set{\pad \in \paths \colon \gr{\forall t \in w} \lumap\gr{\pad\gr{t}} = \lu{z}_{t}} \\
		&= \bigcup_{\lu{z}_w \in \lu{S}} \gr*{ \bigcap_{t \in w} \set{\pad \in \paths \colon \lumap\gr{\pad\gr{t}} = \lu{z}_t} } \\
		&= \bigcup_{\lu{z}_w \in \lu{S}} \gr*{ \bigcap_{t \in w} \left[ \bigcup_{z_t \in \invlumap\gr{\lu{z}_t}} \gr{X_t = z_t} \right] }.
	\end{align*}
	It is now immediately clear that $\invlumap_\paths\gr{\lu{A}_u}$ is an element of $\eventalgebra_u$.
	\qed
\end{proof}

The inverse $\invlumap_\paths$ naturally suggests a sensible formal definition of the lumped stochastic process~$\problum \colon \luprocdomain{} \to \reals$ where, for all $\gr{\lu{A}_u, \lu{X}_u = \lu{x}_u}$ in $\luprocdomain$,
\begin{equation}
\label{eqn:New definition of lumped stochastic process}
	\problum\gr{\lu{A}_u \cond \lu{X}_u = \lu{x}_u}
	\coloneqq \frac{ \sum_{x_u \in \invlumap\gr{\lu{x}_u}} \prob\gr{ \invlumap_\paths\gr{\lu{A}_u} \cond X_u = x_u} \prob\gr{X_u = x_u} }{\sum_{z_u \in \invlumap\gr{\lu{x}_u}} \probm\gr{X_u = z_u}}.
\end{equation}
That this is a proper definition follows from Lemmas~\ref{lem:Irreducible CTMC:Any path has pos prob} and \ref{lem:invlumap_paths is element of eventalgebra}: we know that $\smash{\sum_{z_u \in \invlumap\gr{\lu{x}_u}} \probm\gr{X_u = x_u} > 0}$ by the former and that \smash{$\invlumap_\paths\gr{\lu{A}_u}$} is in $\eventalgebra_u$ by the latter, which in turn implies that $\gr{\invlumap_\paths\gr{\lu{A}_u}, X_u = x_u}$ is in $\procdomain{}$.
Note that if the conditioning event is $\gr{\lu{X}_\emptytseq{} = \lu{x}_\emptytseq}$, then \eqref{eqn:New definition of lumped stochastic process} reduces to
\begin{equation}
\label{eqn:New defintion of lumped stochastic process:Marginal}
	\problum\gr{\lu{A}_\emptytseq}
	\coloneqq \problum\gr{\lu{A}_\emptytseq \cond \lu{X}_\emptytseq = \lu{x}_\emptytseq}
	% = \frac{ \probm\gr{ \invlumap_\paths\gr{\lu{A}_\emptytseq} \cond X_\emptytseq = x_\emptytseq} \probm\gr{X_\emptytseq = x_\emptytseq} }{\probm\gr{X_\emptytseq = z_\emptytseq}}
	= \probm\gr{\invlumap_\paths\gr{\lu{A}_\emptytseq}}.
\end{equation}
One intuitively expects that this definition yields a stochastic process, and this intuition is verified by the following result.
\begin{theorem}
\label{the:Lumped process is a stochastic process}
	If $\probm$ is a positive and irreducible CTMC and $\lumap \colon \stsp \to \lustsp$ a lumping map, then $\problum \colon \luprocdomain{} \to \reals$, as defined by \eqref{eqn:New definition of lumped stochastic process}, is a stochastic process.
\end{theorem}
\begin{proof}
	To prove the stated, we take a little detour.
	First, we combine Definition~\ref{def:HCTMC as stochastic process}, Definition~\ref{def:Stochastic process} and Lemma~\ref{lem:Coherent contional probability can be extended} to see that $\probm$ can be extended to a coherent conditional probability $\probm^{\star}$ on $\ccpevents\gr{\paths} \times \ccpevents\gr{\paths}_{\emptyset}$.
	We take any such coherent extension $\probm^{\star}$, and use it to construct the real-valued map $\problum^{\star}$ on $\ccpevents\gr{\lu\paths} \times \ccpevents\gr{\lu\paths}_{\emptyset}$, defined by
	\begin{equation}
	\label{eqn:definition of problumstar}
		\problum^{\star}\gr{\lu{A} \cond \lu{C}}
		\coloneqq \probm^{\star}\gr{\invlumap_\paths\gr{\lu{A}} \cond \invlumap_\paths\gr{\lu{C}}}
		\quad\text{for all $\gr{\lu{A}, \lu{C}} \in \ccpevents\gr{\lu\paths} \times \ccpevents\gr{\lu\paths}_{\emptyset}$.}
	\end{equation}
	We will show that $\problum$ is a stochastic process by verifying that $\problum^{\star}$ is its (coherent) extension to $\ccpevents\gr{\lu\paths} \times \ccpevents\gr{\lu\paths}_{\emptyset}$, after which we can simply invoke Lemma~\ref{lem:Coherent conditional probability if and only if it can be extended}.

	To that end, we first verify that $\problum^{\star}$ is a coherent conditional probability.
	Therefore, we fix any $n$ in $\nats$, $\gr{\lu{A}_1, \lu{C}_1 }$, $\dots$, $\gr{\lu{A}_n, \lu{C}_n}$ in $\ccpevents\gr{\lu\paths} \times \ccpevents\gr{\lu\paths}_{\emptyset}$ and $\lambda_1$, $\dots$, $\lambda_n$ in $\reals$ and show that $\max S \geq 0$, where
	\[
		S
		\coloneqq \set*{ \sum_{i=1}^{n}\lambda_i \indica{\lu{C}_i}\gr{\lu\pad} \gr*{\problum^{\star}\gr{\lu{A}_i \cond \lu{C}_i} - \indica{\lu{A}_i}\gr{\lu\pad}} \colon \lu\pad \in \bigcup_{i=1}^{n} \lu{C}_i }.
	\]
	Substituting \eqref{eqn:definition of problumstar} yields
	\[
		S
		= \set*{ \sum_{i=1}^{n}\lambda_i \indica{\lu{C}_i}\gr{\lu\pad} \gr*{\probm^{\star}\gr{\invlumap_\paths\gr{\lu{A}_i} \cond \invlumap_\paths\gr{\lu{C}_i}} - \indica{\lu{A}_i}\gr{\lu\pad}} \colon \lu\pad \in \bigcup_{i=1}^{n} \lu{C}_i }.
	\]
	Furthermore, using \eqref{eqn:invlumap_paths} and \eqref{eqn:invlumap_paths is inverse} yields
	\[
		S
		= \set*{ \sum_{i=1}^{n}\lambda_i \indica{\lu{C}_i}\gr{\lumap \circ \pad} \gr*{\probm^{\star}\gr{\invlumap_\paths\gr{\lu{A}_i} \cond \invlumap_\paths\gr{\lu{C}_i}} - \indica{\lu{A}_i}\gr{\lumap \circ \pad}} \colon \pad \in \bigcup_{i=1}^{n} \invlumap_\paths\gr{\lu{C}_i} }.
	\]
	Observe that for all $\pad$ in $\paths$ and $\lu{A} \subseteq \lu\paths$,
	\begin{equation}
	\label{eqn:equality of indicators}
		\indica{\lu{A}}\gr{\lumap\circ\pad}
		= \begin{cases}
			1 &\text{if } \lumap \circ \pad \in \lu{A} \\
			0 &\text{otherwise}
		\end{cases}
		= \begin{cases}
			1 &\text{if } \pad \in \invlumap_\paths\gr{\lu{A}} \\
			0 &\text{otherwise}
		\end{cases}
		= \indica{\invlumap_\paths\gr{\lu{A}}}\gr{\pad},
	\end{equation}
	where the second equality follows immediately from \eqref{eqn:invlumap_paths}.
	We substitute \eqref{eqn:equality of indicators} in our expression for $S$, to yield
	\[
		S
		= \set*{ \sum_{i=1}^{n}\lambda_i \indica{C_i}\gr{\pad} \gr*{\probm^{\star}\gr{A_i \cond C_i} - \indica{A_i}\gr{\pad}} \colon \pad \in \bigcup_{i=1}^{n} C_i },
	\]
	where, for all $i$ in $\set{1, \dots, n}$, we let $A_i \coloneqq \invlumap_\paths\gr{\lu{A}_i}$ and $C_i \coloneqq \invlumap_\paths\gr{\lu{C}_i}$.
	Because $\probm^{\star}$ is a coherent conditional probability on $\ccpevents\gr{\paths} \times \ccpevents\gr{\paths}_{\emptyset}$, it follows from Definition~\ref{def:Coherent conditional probability} that $\max S \geq 0$.

	Next, we verify that $\problum^\star$ coincides with $\problum$ on $\luprocdomain{}$.
	To that end, we fix any $\gr{\lu{A}_u, \lu{X}_u = \lu{x}_u}$ in $\luprocdomain{}$.
	Then by \eqref{eqn:definition of problumstar},
	\[
		\problum^\star\gr{\lu{A}_u \cond \lu{X}_u = \lu{x}_u}
		= \probm^\star\gr{\invlumap_\paths\gr{\lu{A}_u} \cond \invlumap_\paths\gr{\lu{X}_u = \lu{x}_u}}.
	\]
	As $\probm^\star$ is a coherent conditional probability on $\ccpevents\gr{\stsp} \times \ccpevents\gr{\stsp}_{\emptyset}$, it follows from Lemma~\ref{lem:Coherent conditional probability satisfies the laws of probability}\ref{LOP:Bayes rule} that
	\begin{multline}
	\label{eqn:Bayes rule in proof of coherence}
		\probm^\star\gr{\invlumap_\paths\gr{\lu{A}_u} \cond \invlumap_\paths\gr{\lu{X}_u = \lu{x}_u} } \probm^\star \gr{ \invlumap_\paths\gr{\lu{X}_u = \lu{x}_u} } \\
		= \probm^\star\gr{\invlumap_\paths\gr{\lu{A}_u} \cap \invlumap_\paths\gr{\lu{X}_u = \lu{x}_u}},
	\end{multline}
	where $\probm^\star \gr{ \invlumap_\paths\gr{\lu{X}_u = \lu{x}_u} } \coloneqq \probm^\star \gr{ \invlumap_\paths\gr{\lu{X}_u = \lu{x}_u} \cond \paths}$.

	Recall from \eqref{eqn:invlumap_paths:state instantiation} that
	\[
		\invlumap_\paths\gr{\lu{X}_u = \lu{x}_u} = \bigcup_{z_u \in \invlumap\gr{\lu{x}_u}} \gr{X_u = z_u},
	\]
	which clearly is an element of $\eventalgebra_\emptytseq$.
	Consequently, $\gr{\cup_{z_u \in \invlumap\gr{\lu{x}_u}} \gr{X_u = z_u}, \paths}$ is an element of $\procdomain{}$.
	Since furthermore $\probm^\star$ is an extension of $\probm$, we find that
	\begin{equation}
	\label{eqn:Bayes rule in proof of coherence:1}
		\probm^\star \gr{ \invlumap_\paths\gr{\lu{X}_u = \lu{x}_u} }
		= \probm \gr{ \cup_{z_u \in \invlumap\gr{\lu{x}_u}} \gr{X_u = z_u} } = \sum_{z_u \in \invlumap\gr{\lu{x}_u}} \probm \gr{ X_u = z_u }.
	\end{equation}
	Recall from Lemma~\ref{lem:invlumap_paths is element of eventalgebra} that $\invlumap_\paths\gr{\lu{A}_u}$ is an element of $\eventalgebra_u$, such that $\invlumap_\paths\gr{\lu{A}_u} \cap \invlumap_\paths\gr{\lu{X}_u = \lu{x}_u}$ is an element of $\eventalgebra_u$ as well.
	Consequently, we now find that
	\begin{multline}
	\label{eqn:Bayes rule in proof of coherence:2}
		\probm^\star\gr{\invlumap_\paths\gr{\lu{A}_u} \cap \invlumap_\paths\gr{\lu{X}_u = \lu{x}_u}}
		= \probm\gr{\invlumap_\paths\gr{\lu{A}_u} \cap \invlumap_\paths\gr{\lu{X}_u = \lu{x}_u}} \\
		= \probm\gr{\invlumap_\paths\gr{\lu{A}_u} \cap \gr{\cup_{x_u \in \lu{x}_u} \gr{X_u = x_u}}}
		= \sum_{x_u \in \lu{x}_u} \probm\gr{\invlumap_\paths\gr{\lu{A}_u} \cap \gr{X_u = x_u}} \\
		= \sum_{x_u \in \lu{x}_u} \probm\gr{\invlumap_\paths\gr{\lu{A}_u} \cond X_u = x_u} \probm\gr{X_u = x_u}.
	\end{multline}

	Since we know from Lemma~\ref{lem:Irreducible CTMC:Any path has pos prob} that $\sum_{z_u \in \invlumap\gr{\lu{x}_u}} \probm \gr{X_u = z_u} > 0$, substituting \eqref{eqn:Bayes rule in proof of coherence:1} and \eqref{eqn:Bayes rule in proof of coherence:2} in \eqref{eqn:Bayes rule in proof of coherence} yields
	\begin{align*}
		\problum^\star\gr{\lu{A}_u \cond \lu{X}_u = \lu{x}_u}
		&= \probm^\star\gr{\invlumap_\paths\gr{\lu{A}_u} \cond \invlumap_\paths\gr{\lu{X}_u = \lu{x}_u} } \\
		&= \frac{\sum_{x_u \in \invlumap\gr{\lu{x}_u}} \probm\gr{\invlumap_\paths\gr{\lu{A}_u} \cond X_u = x_u} \probm\gr{X_u = x_u}}{\sum_{z_u \in \invlumap\gr{\lu{x}_u}} \probm \gr{ X_u = z_u}} \\
		&= \problum\gr{\lu{A}_u \cond \lu{X}_u = \lu{x}_u}.
	\end{align*}
	Hence, the coherent conditional probability $\problum^\star$ on $\ccpevents\gr{\lustsp} \times \ccpevents\gr{\lustsp}_{\emptyset}$ coincides with the real-valued map $\problum$ on $\luprocdomain{}$.
	It now follows from Lemma~\ref{lem:Coherent conditional probability if and only if it can be extended} that $\problum$ is a coherent conditional probability on $\luprocdomain{}$, such that it is a stochastic process by Definition~\ref{def:Stochastic process}.
	\qed
\end{proof}

\begin{lemma}
\label{lem:Lumped process:One step conditional}
	If $\probm$ is a positive and irreducible CTMC and $\lumap \colon \stsp \to \lustsp$ a lumping map, then for all $t$ in $\nnegreals$, $u$ in $\setoftseql{t}$, $\lu{y}$ in $\lustsp$ and $\lu{x}_u$ in $\lustsp_u$,
	\begin{multline}
	\label{eqn:Lumped process:One step conditional}
		\problum\gr{\lu{X}_t = \lu{y} \cond \lu{X}_u = \lu{x}_u} \\
		= \frac{
			\sum_{x_u \in \invlumap\gr{\lu{x}_u}} \sum_{y \in \invlumap\gr{\lu{y}}} \probm\gr{X_t = y \cond X_u = x_u} \probm\gr{X_u = x_u}
		}{
			\sum_{z_u \in \invlumap\gr{\lu{x}_u}} \probm\gr{X_u = z _u}
		}.
	\end{multline}
\end{lemma}
\begin{proof}
	Follows immediately from \eqref{eqn:invlumap_paths:state instantiation} and \eqref{eqn:New definition of lumped stochastic process}.
	\qed
\end{proof}

	\subsection{The Instantaneous Transition Rate Matrix of the Lumped Stochastic Process}
For any $t$ in $\nnegreals$, $u$ in $\setoftseql{t}$ and $\lu{x}_u$ in $\lustsp_u$, we consider the real-valued map~$\dist_{\gr{u, \lu{x}_u, t}}$ that maps any $x$ in $\stsp$ to
\begin{equation}
\label{eqn:dist u x_u t}
	% \dist_{\gr{u, \lu{x}_u, t}}
	% \colon \stsp \to \nnegreals
	% \colon x \mapsto
	\dist_{\gr{u, \lu{x}_u, t}}\gr{x}
	\coloneqq \frac{
			\sum_{x_u \in \invlumap\gr{\lu{x}_u}} \probm\gr{X_u = x_u, X_t = x }
		}{
			\sum_{z_u \in \invlumap\gr{\lu{x}_u}} \probm\gr{X_u = z_u}
		}.
\end{equation}
We use this notation in the following result, which provides the main motivation for seeing the lumped process as belonging to an imprecise CTMC that is consistent with a specific set of transition rate matrices.
\begin{proposition}
\label{prop:Instantaneous transition rate of the lumped stochastic process}
	If $\markov$ is a positive and irreducible CTMC and $\lumap \colon \stsp \to \lustsp$ a lumping map, then, for all $t$ in $\nnegreals{}$, $u$ in $\setoftseql{t}$, $\lu{x}, \lu{y}$ in $\lustsp$ and $\lu{x}_{u}$ in $\lustsp_{u}$,
	\begin{multline}
	\label{eqn:inst trm:from the right}
		\lim_{\Delta \to 0^{+}} \frac{1}{\Delta} \gr*{\problum\gr{ \lu{X}_{t+\Delta} = \lu{y} \cond \lu{X}_{u} = \lu{x}_{u}, \lu{X}_{t} = \lu{x} } - \indica{\lu{x}}\gr{\lu{y}} } \\
		= \sum_{x \in \invlumap\gr{\lu{x}}} \frac{
			\dist_{\gr{u, \lu{x}_{u}, t}} \gr{x}
		}{
			\sum_{z \in \invlumap\gr{\lu{x}}} \dist_{\gr{u, \lu{x}_{u}, t}} \gr{z}
		} \sum_{y \in \invlumap\gr{\lu{y}}} \trm\gr{x, y},
	\end{multline}
	which is a convex combination of terms $\sum_{y \in \invlumap\gr{\lu{y}}} \trm\gr{x, y}$ with $x$ in $\invlumap\gr{\lu{x}}$, and, if $t \neq 0$, also
	\begin{multline}
	\label{eqn:inst trm:from the left}
		\lim_{\Delta \to 0^{+}} \frac{1}{\Delta} \gr*{ \problum\gr{ \lu{X}_{t} = \lu{y} \cond \lu{X}_{u} = \lu{x}_{u}, \lu{X}_{t - \Delta} = \lu{x} } - \indica{\lu{x}}\gr{\lu{y}} } \\
		= \sum_{x \in \invlumap\gr{\lu{x}}} \frac{
			\dist_{\gr{u, \lu{x}_{u}, t}} \gr{x}
		}{
			\sum_{z \in \invlumap\gr{\lu{x}}} \dist_{\gr{u, \lu{x}_{u}, t}} \gr{z}
		} \sum_{y \in \invlumap\gr{\lu{y}}} \trm\gr{x, y}.
	\end{multline}
\end{proposition}
Before proving this result, we first state and, if necessary, prove three intermediary technical results.
\begin{lemma}
\label{lem:The fraction is a positive distribution}
	Let $\markov$ be a positive and irreducible CTMC.
	If $t$ in $\nnegreals$, $u$ in $\setoftseql{t}$ and $\lu{x}_{u}$ in $\lustsp_{u}$, then $\dist_{\gr{u, \lu{x}_u, t}}$ is a positive distribution on $\stsp$.
\end{lemma}
\begin{proof}
	For any $x$ in $\stsp$,
	\[
		\dist_{\gr{u, \lu{x}_u, t}}\gr{x}
		= \frac{
			\sum_{x_u \in \invlumap\gr{\lu{x}_u}} \probm\gr{X_{u} = x_u, X_{t} = x}
		}{
			\sum_{z_u \in \invlumap\gr{\lu{x}_u}} \probm\gr{X_{u} = z_u}
		}
	\]
	is a well-defined (in the sense that we do not divide by zero) positive real number due to Lemma~\ref{lem:Irreducible CTMC:Any path has pos prob}.
	Hence, since
	\begin{align*}
		\sum_{x \in \stsp} \dist_{\gr{u, \lu{x}_u, t}}\gr{x}
		&= \sum_{x \in \stsp} \frac{
			\sum_{x_u \in \invlumap\gr{\lu{x}_u}} \probm\gr{X_u = x_u, X_t = x }
		}{
			\sum_{x_u \in \invlumap\gr{\lu{x}_u}} \probm\gr{X_u = x_u}
		} \\
		&= \sum_{x \in \stsp}\frac{
			\sum_{x_u \in \invlumap\gr{\lu{x}_u}} \probm\gr{X_{u} = x_u, X_{t} = x}
		}{
			\sum_{z \in \stsp} \sum_{x_u \in \invlumap\gr{\lu{x}_u}} \probm\gr{X_{u} = x_u, X_{t} = z}
		}
		= 1,
	\end{align*}
	$\dist_{\gr{u, \lu{x}_u, t}}$ is a positive distribution on $\stsp$.
	\qed
\end{proof}

\begin{lemma}[Theorem~2.1.1 in \cite{1997Norris}]
\label{lem:trm as limit}
	Let $\trm$ be a transition rate matrix.
	Then for all $t$ in $\nnegreals$ and all $x, y$ in $\stsp$,
	\[
		\lim_{\Delta \to 0^+} \frac{\tm_{t+\Delta}\gr{x, y} - \tm_{t}\gr{x, y}}{\Delta}
		= [\trm \tm_t] \gr{x, y}
	\]
	and, if $t \neq 0$,
	\[
		\lim_{\Delta \to 0^+} \frac{\tm_{t}\gr{x, y} - \tm_{t-\Delta}\gr{x, y}}{\Delta}
		= [\trm \tm_t] \gr{x, y}.
	\]
\end{lemma}

\begin{lemma}
\label{lem:dist_u luxu t is continuouos in t}
	If $\probm$ is a positive and irreducible CTMC and $\lumap \colon \stsp \to \lustsp$ a lumping map, then, for all $t$ in $\posreals$, $u$ in $\setoftseql{t}$, $\lu{x}$ in $\lustsp$ and $x$ in $\stsp$,
	\[
		\lim_{\Delta \to 0^+} \dist_{\gr{u, \lu{x}_u, t - \Delta}}\gr{x}
		= \dist_{\gr{u, \lu{x}_u, t}}\gr{x}.
	\]
\end{lemma}
\begin{proof}
	Fix some $\Delta$ in $\posreals$ such that $\Delta\leq t$ and $\max u < t - \Delta$.
	Then by \eqref{eqn:dist u x_u t},
	\[
		\dist_{\gr{u, \lu{x}_u, t-\Delta}}\gr{x}
		= \frac{
				\sum_{x_u \in \invlumap\gr{\lu{x}_u}} \probm\gr{X_u = x_u, X_{t-\Delta} = x }
			}{
				\sum_{z_u \in \invlumap\gr{\lu{x}_u}} \probm\gr{X_u = z_u}
			}.
	\]
	For notational simplicity, we distinguish between two cases.

	First, we assume that $u = t_1, \dots, t_n$ is not the empty sequence $\emptytseq{}$.
	Then using \eqref{eqn:HCTMC Markov}, \eqref{eqn:HCTMC Homogeneity} and \eqref{eqn:HCTMC transition}, we can write the numerator as
	\begin{multline*}
		\sum_{x_u \in \invlumap\gr{\lu{x}_u}} \probm\gr{X_u = x_u, X_{t-\Delta} = x }
		= \sum_{x_u \in \invlumap\gr{\lu{x}_u}} \probm\gr{X_u = x_u} \probm\gr{X_{t-\Delta} = x \cond X_u = x_u } \\
		= \sum_{x_u \in \invlumap\gr{\lu{x}_u}} \probm\gr{X_u = x_u} \probm\gr{X_{t-\Delta} = x \cond X_{t_n} = x_{t_n} } \\
		= \sum_{x_u \in \invlumap\gr{\lu{x}_u}} \probm\gr{X_u = x_u} \probm\gr{X_{t-\Delta-t_n} = x \cond X_{0} = x_{t_n} } \\
		= \sum_{x_u \in \invlumap\gr{\lu{x}_u}} \probm\gr{X_u = x_u} \tm_{t-\Delta-t_n}\gr{x_{t_n}, x}.
	\end{multline*}
	Since it follows from Lemma~\ref{lem:trm as limit} that $\lim_{\Delta \to 0^+} \tm_{t-\Delta-t_n}\gr{y, x} = \tm_{t-t_n}\gr{y, x}$ for all $y$ in $\stsp$, we find that
	\begin{align*}
		\lim_{\Delta \to 0^+} \dist_{\gr{u, \lu{x}_u, t-\Delta}}\gr{x}
		&= \frac{
			\sum_{x_u \in \invlumap\gr{\lu{x}_u}} \lim_{\Delta \to 0^+} \probm\gr{X_u = x_u} \tm_{t-\Delta-t_n}\gr{x_{t_n}, x}
		}{
			\sum_{z_u \in \invlumap\gr{\lu{x}_u}} \probm\gr{X_u = z_u}
		} \\
		&= \frac{
			\sum_{x_u \in \invlumap\gr{\lu{x}_u}} \probm\gr{X_u = x_u} \tm_{t-t_n}\gr{x_{t_n}, x}
		}{
			\sum_{z_u \in \invlumap\gr{\lu{x}_u}} \probm\gr{X_u = z_u}
		}.
	\end{align*}
	Executing the same manipulations as before in reverse order yields
	\[
		\lim_{\Delta \to 0^+} \dist_{\gr{u, \lu{x}_u, t-\Delta}}\gr{x}
		= \dist_{\gr{u, \lu{x}_u, t}}\gr{x}.
	\]

	Next, we assume that $u$ is the empty sequence $\emptytseq{}$.
	Then some straightforward manipulations yield
	\begin{align*}
		\dist_{\gr{u, \lu{x}_u, t-\Delta}}\gr{x}
		&= \frac{\probm\gr{X_{\emptytseq{}} = x_{\emptytseq{}}, X_{t-\Delta} = x}}{\probm\gr{X_{\emptytseq{}} = x_{\emptytseq{}}}}
		= \probm\gr{X_{t-\Delta} = x} \\
		&= \sum_{y \in \stsp} \probm\gr{X_0 = y} \probm\gr{X_{t-\Delta} = x \cond X_0 = y}
		= \sum_{y \in \stsp} \inidist\gr{y} \tm_{t-\Delta}\gr{y, x}.
	\end{align*}
	Again, it now follows from Lemma~\ref{lem:trm as limit} that
	\[
		\lim_{\Delta \to 0^+} \dist_{\gr{\emptytseq, \lu{x}_{u}, t-\Delta}}\gr{x}
		= \sum_{y \in \stsp} \inidist\gr{y} \lim_{\Delta \to 0^+} \tm_{t-\Delta}\gr{y, x}
		= \sum_{y \in \stsp} \inidist\gr{y} \tm_{t}\gr{y, x},
	\]
	such that
	\[
		\lim_{\Delta \to 0^+} \dist_{\gr{u, \lu{x}_u, t-\Delta}}\gr{x}
		= \dist_{\gr{u, \lu{x}_u, t}}\gr{x}
	\]
	\qed
\end{proof}

\begin{lemma}
\label{lem:Transition probability of the lumped stochastic process}
	Let $\markov$ be a positive and irreducible CTMC, $\lumap \colon \stsp \to \lustsp$ a lumping map and $\problum$ the corresponding lumped stochastic process.
	Fix any $t, \Delta$ in $\nnegreals{}$, $u$ in $\setoftseql{t}$, $\lu{x}, \lu{y}$ in $\lustsp$ and $\lu{x}_{u}$ in $\lustsp_{u}$.
	Then
	\begin{multline}
	\label{eqn:multiline thing}
		\problum\gr{ \lu{X}_{t+\Delta} = \lu{y} \cond \lu{X}_{u} = \lu{x}_{u}, \lu{X}_{t} = \lu{x} } \\
		= \sum_{x \in \invlumap\gr{\lu{x}}}
		\frac{
			\dist_{\gr{u, \lu{x}_{u}, t}} \gr{x}
		}{
			\sum_{z \in \invlumap\gr{\lu{x}}} \dist_{\gr{u, \lu{x}_{u}, t}} \gr{z}
		} \sum_{y \in \invlumap\gr{\lu{y}}} \tm_{\Delta}\gr{x, y},
	\end{multline}
	which is a convex combination of terms $\sum_{y \in \invlumap\gr{\lu{y}}} \tm_{\Delta}\gr{x, y}$ with $x$ in $\invlumap\gr{\lu{x}}$.
	If moreover $\Delta \leq t$ and $\max u < t - \Delta$, then
	\begin{multline}
	\label{eqn:multiline thing:from the left}
		\problum\gr{ \lu{X}_{t} = \lu{y} \cond \lu{X}_{u} = \lu{x}_{u}, \lu{X}_{t-\Delta} = \lu{x} } \\
		= \sum_{x \in \invlumap\gr{\lu{x}}}
		\frac{
			\dist_{\gr{u, \lu{x}_{u}, t-\Delta}} \gr{x}
		}{
			\sum_{z \in \invlumap\gr{\lu{x}}} \dist_{\gr{u, \lu{x}_{u}, t-\Delta}} \gr{z}
		} \sum_{y \in \invlumap\gr{\lu{y}}} \tm_{\Delta}\gr{x, y},
	\end{multline}
\end{lemma}
\begin{proof}
	We only prove the first equality because the proof of the second equality is largely analoguous.
	To that end, we recall that by \eqref{eqn:Lumped process:One step conditional},
	\begin{multline*}
		\problum\gr{ \lu{X}_{t+\Delta} = \lu{y} \cond \lu{X}_{u} = \lu{x}_{u}, \lu{X}_{t} = \lu{x} } \\
		= \frac{
			\sum\limits_{x_{u} \in \invlumap\gr{\lu{x}_{u}}}
			\sum\limits_{x \in \invlumap\gr{\lu{x}}}
			\sum\limits_{y \in \invlumap\gr{\lu{y}}}
			\probm\gr{X_{u} = x_{u}, X_{t} = x, X_{t+\Delta} = y}
		}{
			\sum\limits_{z_{u} \in \invlumap\gr{\lu{x}_{u}}}
			\sum\limits_{z \in \invlumap\gr{\lu{x}}}
			\probm\gr{ X_{u} = z_{u}, X_{t} = z }
		}.
	\end{multline*}
	After applying \eqref{eqn:HCTMC Markov}--\eqref{eqn:HCTMC transition} and reordering the sums, we end up with
	\begin{multline*}
		\problum\gr{ \lu{X}_{t+\Delta} = \lu{y} \cond \lu{X}_{u} = \lu{x}_{u}, \lu{X}_{t} = \lu{x} } \\
		= \frac{
			\sum\limits_{x \in \invlumap\gr{\lu{x}}}
			\sum\limits_{y \in \invlumap\gr{\lu{y}}}
			\tm_{\Delta}\gr{x, y}
			\sum\limits_{x_{u} \in \invlumap\gr{\lu{x}_{u}}}
			\probm\gr{ X_{u} = x_{u}, X_{t} = x }
		}{
			\sum\limits_{z_{u} \in \invlumap\gr{\lu{x}_{u}}}
			\sum\limits_{z \in \invlumap\gr{\lu{x}}}
			\probm\gr{ X_{u} = z_u, X_{t} = z }
		}.
	\end{multline*}
	It is a matter of straightforward verification that if we divide both the numerator and the denominator in this final expression by $\sum_{x_u \in \invlumap\gr{\lu{x}_u}} \probm\gr{X_u = x_u}$, then the obtained expression is indeed equal to \eqref{eqn:multiline thing}.
	Finally, verifying that \eqref{eqn:multiline thing} is a convex combination is trivial because $\dist_{\gr{u, \lu{x}_u, t}}$ is a positive distribution on $\stsp$ by Lemma~\ref{lem:The fraction is a positive distribution}.
	\qed
\end{proof}

\begin{proofof}{Proposition~\ref{prop:Instantaneous transition rate of the lumped stochastic process}}
	We start by proving \eqref{eqn:multiline thing}, i.e., the limit from the right.
	To that end, we fix any $\Delta$ in $\posreals$ and recall that by Lemma~\ref{lem:Transition probability of the lumped stochastic process},
	\[
		\problum\gr{ \lu{X}_{t+\Delta} = \lu{y} \cond \lu{X}_{u} = \lu{x}_{u}, \lu{X}_{t} = \lu{x} } \\
		= \sum_{x \in \invlumap\gr{\lu{x}}}
		\frac{
			\dist_{\gr{u, \lu{x}_{u}, t}} \gr{x}
		}{
			\sum_{z \in \invlumap\gr{\lu{x}}} \dist_{\gr{u, \lu{x}_{u}, t}} \gr{z}
		} \sum_{y \in \invlumap\gr{\lu{y}}} \tm_{\Delta}\gr{x, y},
	\]
	where $\dist_{\gr{u, \lu{x}_{u}, t}}$ is a positive distribution on $\stsp$ by Lemma~\ref{lem:The fraction is a positive distribution}.
	Subtracting $\indica{\lu{x}}\gr{\lu{y}}$ from both sides of the equality and dividing both sides of the equality by $\Delta$ yields
	\begin{multline*}
		\frac{
			\problum\gr{ \lu{X}_{t+\Delta} = \lu{y} \cond \lu{X}_{u} = \lu{x}_{u}, \lu{X}_{t} = \lu{x} } - \indica{\lu{x}}\gr{\lu{y}}
		}{
			\Delta
		} \\
		= \frac{1}{\Delta} \gr*{ \gr*{\sum_{x \in \invlumap\gr{\lu{x}}}
		\frac{
			\dist_{\gr{u, \lu{x}_{u}, t}} \gr{x}
		}{
			\sum_{z \in \invlumap\gr{\lu{x}}} \dist_{\gr{u, \lu{x}_{u}, t}} \gr{z}
		} \sum_{y \in \invlumap\gr{\lu{y}}} \tm_{\Delta}\gr{x, y} } - \indica{\lu{x}}\gr{\lu{y}} }.
	\end{multline*}
	Recall that the sum for $x$ ranging over $\invlumap\gr{\lu{x}}$ is a convex combination of the terms $\sum_{y \in \invlumap\gr{\lu{y}}} \tm_{\Delta}\gr{x, y}$, such that we can rewrite this equality as
	\begin{multline*}
		\frac{
			\problum\gr{ \lu{X}_{t+\Delta} = \lu{y} \cond \lu{X}_{u} = \lu{x}_{u}, \lu{X}_{t} = \lu{x} } - \indica{\lu{x}}\gr{\lu{y}}
		}{
			\Delta
		} \\
		= \frac{1}{\Delta} \gr*{ \sum_{x \in \invlumap\gr{\lu{x}}}
		\frac{
			\dist_{\gr{u, \lu{x}_{u}, t}} \gr{x}
		}{
			\sum_{z \in \invlumap\gr{\lu{x}}} \dist_{\gr{u, \lu{x}_{u}, t}} \gr{z}
		} \gr*{ \gr*{ \sum_{y \in \invlumap\gr{\lu{y}}} \tm_{\Delta}\gr{x, y} } - \indica{\lu{x}}\gr{\lu{y}} } }.
	\end{multline*}
	Furthermore, it clearly holds that $\indica{\lu{x}}\gr{\lu{y}} = \sum_{y \in \invlumap\gr{\lu{y}}} \indica{x}\gr{y}$, such that
	\begin{multline*}
		\frac{
			\problum\gr{ \lu{X}_{t+\Delta} = \lu{y} \cond \lu{X}_{u} = \lu{x}_{u}, \lu{X}_{t} = \lu{x} } - \indica{\lu{x}}\gr{\lu{y}}
		}{
			\Delta
		} \\
		= \frac{1}{\Delta} \gr*{ \sum_{x \in \invlumap\gr{\lu{x}}}
		\frac{
			\dist_{\gr{u, \lu{x}_{u}, t}} \gr{x}
		}{
			\sum_{z \in \invlumap\gr{\lu{x}}} \dist_{\gr{u, \lu{x}_{u}, t}} \gr{z}
		} \sum_{y \in \invlumap\gr{\lu{y}}} \gr*{\tm_{\Delta}\gr{x, y}  - \indica{x}\gr{y} } }.
	\end{multline*}
	Since $T_0\gr{x, y} = I\gr{x, y} = \indica{x}\gr{y}$, it follows from Lemma~\ref{lem:trm as limit} that
	\begin{multline*}
		\lim_{\Delta \to 0^+} \frac{
			\problum\gr{ \lu{X}_{t+\Delta} = \lu{y} \cond \lu{X}_{u} = \lu{x}_{u}, \lu{X}_{t} = \lu{x} } - \indica{\lu{x}}\gr{\lu{y}}
		}{
			\Delta
		} \\
		= \sum_{x \in \invlumap\gr{\lu{x}}}
		\frac{
			\dist_{\gr{u, \lu{x}_{u}, t}} \gr{x}
		}{
			\sum_{z \in \invlumap\gr{\lu{x}}} \dist_{\gr{u, \lu{x}_{u}, t}} \gr{z}
		} \sum_{y \in \invlumap\gr{\lu{y}}} \lim_{\Delta \to 0^+} \frac{\tm_{\Delta}\gr{x, y}  - \indica{x}\gr{y}}{\Delta} \\
		= \sum_{x \in \invlumap\gr{\lu{x}}}
		\frac{
			\dist_{\gr{u, \lu{x}_{u}, t}} \gr{x}
		}{
			\sum_{z \in \invlumap\gr{\lu{x}}} \dist_{\gr{u, \lu{x}_{u}, t}} \gr{z}
		} \sum_{y \in \invlumap\gr{\lu{y}}} \trm\gr{x, y}.
	\end{multline*}

	Next, we prove \eqref{eqn:inst trm:from the left}, i.e., the limit from the left.
	To that end, we use Lemma~\ref{lem:Transition probability of the lumped stochastic process} and execute similar manipulations as in the first part of the proof, to yield
	\begin{multline*}
		\lim_{\Delta \to 0^+} \frac{
			\problum\gr{ \lu{X}_{t} = \lu{y} \cond \lu{X}_{u} = \lu{x}_{u}, \lu{X}_{t-\Delta} = \lu{x} } - \indica{\lu{x}}\gr{\lu{y}}
		}{
			\Delta
		} \\
		= \lim_{\Delta \to 0^+} \frac{1}{\Delta} \gr*{ \sum_{x \in \invlumap\gr{\lu{x}}}
		\frac{
			\dist_{\gr{u, \lu{x}_{u}, t-\Delta}} \gr{x}
		}{
			\sum_{z \in \invlumap\gr{\lu{x}}} \dist_{\gr{u, \lu{x}_{u}, t-\Delta}} \gr{z}
		} \sum_{y \in \invlumap\gr{\lu{y}}} \gr*{\tm_{\Delta}\gr{x, y}  - \indica{x}\gr{y} } }.
	\end{multline*}
	It now follows from Lemmas~\ref{lem:trm as limit} and \ref{lem:dist_u luxu t is continuouos in t} that
	\begin{multline*}
		\lim_{\Delta \to 0^+} \frac{
			\problum\gr{ \lu{X}_{t} = \lu{y} \cond \lu{X}_{u} = \lu{x}_{u}, \lu{X}_{t-\Delta} = \lu{x} } - \indica{\lu{x}}\gr{\lu{y}}
		}{
			\Delta
		} \\
		= \sum_{x \in \invlumap\gr{\lu{x}}}
		\frac{
			\lim_{\Delta \to 0^+} \dist_{\gr{u, \lu{x}_{u}, t-\Delta}} \gr{x}
		}{
			\sum_{z \in \invlumap\gr{\lu{x}}} \lim_{\Delta \to 0^+} \dist_{\gr{u, \lu{x}_{u}, t-\Delta}} \gr{z}
		} \sum_{y \in \invlumap\gr{\lu{y}}} \lim_{\Delta \to 0^+} \frac{\tm_{\Delta}\gr{x, y}  - \indica{x}\gr{y} }{\Delta} \\
		= \sum_{x \in \invlumap\gr{\lu{x}}}
		\frac{
			\dist_{\gr{u, \lu{x}_{u}, t}} \gr{x}
		}{
			\sum_{z \in \invlumap\gr{\lu{x}}} \dist_{\gr{u, \lu{x}_{u}, t}} \gr{z}
		} \sum_{y \in \invlumap\gr{\lu{y}}} \trm\gr{x, y}.
	\end{multline*}
	\qed
\end{proofof}

The following corollary essentially allows us to use the results from \cite{2017KrakDeBock}.
\begin{corollary}
\label{cor:Lumped process is well-behaved}
	If $\prob$ is a positive and irreducible CTMC and $\lumap \colon \stsp \to \lustsp$ a lumping map, then the corresponding lumped  process $\problum$ is \emph{well-behaved} \cite[Definition~4.4]{2017KrakDeBock}, in the sense that, for all $t$ in $\nnegreals$, $u$ in $\setoftseql{t}$, $x, y$ in $\lustsp$ and $\lu{x}_u$ in $\lustsp_u$,
	\[
		\limsup_{\Delta \to 0^+} \frac{1}{\Delta} \abs*{
			\problum\gr{ \lu{X}_{t+\Delta} = \lu{y} \cond \lu{X}_u = \lu{x}, \lu{X}_t = \lu{x}} - \indica{\lu{x}}\gr{\lu{y}}
		} < + \infty
	\]
	and, if $t \neq 0$,
	\[
		\limsup_{\Delta \to 0^+} \frac{1}{\Delta} \abs*{
			\problum\gr{ \lu{X}_{t} = \lu{y} \cond \lu{X}_u = \lu{x}, \lu{X}_{t - \Delta} = \lu{x} } - \indica{\lu{x}}\gr{\lu{y}}
		} < + \infty.
	\]
\end{corollary}
\begin{proof}
	Follows immediately from Proposition~\ref{prop:Instantaneous transition rate of the lumped stochastic process}.
	\qed
\end{proof}

	\section{The Induced Imprecise Continuous-Time Markov Chains}
\label{app:Induced iCTMC}
Everything is now set up to characterise the imprecise CTMC induced by lumping.
Recall from Appendix~\ref{app:iCTMCs} that such an imprecise CTMC is fully characterised by a non-empty bounded set of transition rate matrices and a non-empty set of initial distributions.
Therefore, we first focus on the set of \emph{lumped transition rate matrices}.

\subsection{The Set of Lumped Transition Rate Matrices}
\label{sapp:Induced iCTMC:Set of lumped trms}
Let $\trm$ be an irreducible transition rate matrix and $\lumap \colon \stsp \to \lustsp$ a lumping map.
Then for any $\dist$ in $\setofposdist\gr{\stsp}$, the matrix~$\lutrm_{\dist} \colon \setoffn\gr{\lustsp} \to \setoffn\gr{\lustsp}$ is defined by
\begin{equation}
\label{eqn:lutrm_dist defintion}
	\lutrm_{\dist}\gr{\lu{x}, \lu{y}}
	\coloneqq \sum_{x \in \invlumap\gr{\lu{x}}} \frac{\dist\gr{x}}{\sum_{z \in \invlumap\gr{\lu{x}}} \dist\gr{z}} \sum_{y \in \invlumap\gr{\lu{y}}} \trm\gr{x, y}
	\text{ for all $\lu{x}, \lu{y}$ in $\lustsp$}.
\end{equation}
\begin{lemma}
\label{lem:lutrm_posdist is a trm and irreducible}
	If $\trm$ is an irreducible transition rate matrix on $\stsp$, $\lumap \colon \stsp \to \lustsp$ a lumping map and $\dist$ an element of $\setofposdist\gr{\stsp}$, then $\lutrm_{\dist}$, defined by \eqref{eqn:lutrm_dist defintion}, is an irreducible transition rate matrix.
\end{lemma}
\begin{proof}
	We first verify that $\lutrm_{\dist}$ is indeed a transition rate matrix on $\lustsp$.
	To that end, we observe that $\lutrm_{\dist}$ is a real-valued $\card{\lustsp} \times \card{\lustsp}$ matrix.
	We also need to verify that $\lutrm_{\dist}$ has non-negative off-diagonal elements and rows that sum up to zero.
	Note that for any $\lu{x}, \lu{y}$ in $\lustsp$ such that $\lu{x} \neq \lu{y}$, $\lutrm_{\dist}\gr{\lu{x}, \lu{y}}$ is a convex combination of non-negative real numbers, and hence the off-diagonal elements are non-negative.
	Also note that for any $\lu{x}$ in $\lustsp$,
	\begin{align*}
		\sum_{\lu{y} \in \lustsp} \lutrm_{\dist}\gr{\lu{x}, \lu{y}}
		&= \sum_{\lu{y} \in \lustsp} \sum_{x \in \invlumap\gr{\lu{x}}} \frac{\dist\gr{x}}{\sum_{z \in \invlumap\gr{\lu{x}}} \dist\gr{z}} \sum_{y \in \invlumap\gr{\lu{y}}} \trm\gr{x, y} \\
		&= \sum_{x \in \invlumap\gr{\lu{x}}} \frac{\dist\gr{x}}{\sum_{z \in \invlumap\gr{\lu{x}}} \dist\gr{z}} \sum_{\lu{y} \in \lustsp} \sum_{y \in \invlumap\gr{\lu{y}}} \trm\gr{x, y} \\
		&= \sum_{x \in \invlumap\gr{\lu{x}}} \frac{\dist\gr{x}}{\sum_{z \in \invlumap\gr{\lu{x}}} \dist\gr{z}} \sum_{y \in \stsp} \trm\gr{x, y}
		= 0,
	\end{align*}
	where the second and third equality follow from manipulations of finite sums and the last equality holds because $\trm$ is a transition rate matrix.

	Next, we prove that $\lutrm_{\dist}$ is irreducible.
	To that end, we fix any two $\lu{x}, \lu{y}$ in $\lustsp$ such that $\lu{x} \neq \lu{y}$.
	Fix now any $x$ in $\invlumap{\gr{\lu{x}}}$ and any $y$ in $\invlumap\gr{\lu{y}}$.
	Then as $\trm$ is irreducible, it follows from Proposition~\ref{prop:Precise irreducibility} that there is a sequence $x_{0}, \dots, x_{n}$ in $\stsp$ such that $x_{0} = x$, $x_{n} = y$ and $\trm\gr{x_{i-1}, x_{i}} > 0$ for all $i = 1, \dots, n$.
	If for all $i$ in $\{ 0, \dots, n \}$ we let $\lu{x}_{i} \coloneqq \lumap\gr{x_{i}}$, then $\lu{x}_{0}, \dots, \lu{x}_{n}$ is obviously a sequence in $\lustsp$ such that $\lu{x}_{0} = \lu{x}$ and $\lu{x}_{n} = \lu{y}$.
	It may occur for several indices $j$ in $\{ 0, \dots, n-1 \}$ that there are consecutive entries $\lu{x}_{j}, \lu{x}_{j+1}, \dots$ that are all equal to $\lu{x}_{j}$.
	For each of those indices $j$ we delete these consecutive entries $\lu{x}_{j+1}, \dots$ from the sequence; this way, we end up with the shorter sequence $\lu{x}_{i_0}, \dots, \lu{x}_{i_m}$ in $\smash{\lustsp}$, where $\{ i_0, \dots, i_m \}$ is an increasing subsequence of $\{ 1, \dots, n \}$.
	Note that by construction $\lu{x}_{i_0} = \lu{x}$, $\lu{x}_{i_m} = \lu{y}$ and $\lu{x}_{i_{\gr{k-1}}} \neq \lu{x}_{i_k}$ for all $k$ in $\{ 1, \dots, m \}$.
	Fix now any $k$ in $\{ 1, \dots, m \}$.
	While it does not necessarily hold that $\trm\gr{x_{i_{\gr{k-1}}}, x_{i_k}} > 0$, we have removed the consecutive entries in such a way that $\trm\gr{x_{i_k-1}, x_{i_k}} > 0$.
	Because clearly $x_{i_k - 1} \in \invlumap\gr{\lu{x}_{i_{\gr{k-1}}}}$ and $x_{i_k} \in \invlumap\gr{\lu{x}_{i_k}}$, it now follows that
	\[
		\lutrm_{\dist}\gr{\lu{x}_{i_{\gr{k-1}}}, \lu{x}_{i_k}}
		= \sum_{x \in \invlumap\gr{\lu{x}_{i_{\gr{k-1}}}}} \frac{\dist\gr{x}}{\sum_{z \in \invlumap\gr{\lu{x}_{i_{\gr{k-1}}}}} \dist\gr{z}} \sum_{y \in \invlumap\gr{\lu{x}_{i_k}}} \trm\gr{x, y}
		> 0.
	\]
	Since this is true for any $\lu{x}, \lu{y}$ in $\lustsp$ such that $\lu{x} \neq \lu{y}$, it follows from Proposition~\ref{prop:Precise irreducibility} that $\lutrm_{\dist}$ is irreducible.
	\qed
\end{proof}

Consider again an irreducible transition rate matrix $\trm$ and a lumping map $\lumap \colon \lustsp{} \to \stsp{}$.
The associated set of lumped transition rate matrices
\begin{equation}
\label{eqn:setoflutrm}
	\setoflutrm{}
	\coloneqq \set*{ \lutrm_{\dist} \colon \dist \in \setofposdist\gr{\stsp} }
	\subseteq \setofalltrm{}\gr{\lustsp}
\end{equation}
plays a vital role in obtaining our imprecise CTMC.
In the remainder of this section, we are only concerned with some of its nice technical properties.
Our proof for one of these properties requires the following lemma.
\begin{lemma}
\label{lem:dist_alpha}
	If $\dist_1$ and $\dist_2$ are two positive distributions on $\stsp$, $\alpha$ is a real number in the open unit interval $(0, 1)$ and $\lumap$ is a lumping map, then $\dist_\alpha$ in $\setoffn\gr{\stsp}$, defined for all $x$ in $\stsp$ as
	\begin{equation}
	\label{eqn:dist_alpha}
		\dist_\alpha\gr{x}
		\coloneqq \frac{
			\alpha \gr*{\sum_{z \in \invlumap\gr{\lumap\gr{x}}} \dist_2\gr{z}} \dist_1\gr{x}
			+ \gr{1-\alpha} \gr*{\sum_{z \in \invlumap\gr{\lumap\gr{x}}} \dist_1\gr{z}} \dist_2\gr{x}
		}{
			\sum_{\lu{y} \in \lustsp} \gr*{\sum_{y \in \invlumap\gr{\lu{y}}} \dist_1\gr{y}} \gr*{\sum_{y \in \invlumap\gr{\lu{y}}} \dist_2\gr{y}}
		}
	\end{equation}
	is a positive distribution on $\stsp$.
\end{lemma}
\begin{proof}
	To reduce the notational burden in the remainder, we define
	\[
		c
		\coloneqq \sum_{\lu{y} \in \lustsp} \gr*{\sum_{y \in \invlumap\gr{\lu{y}}} \dist_1\gr{y}} \gr*{\sum_{y \in \invlumap\gr{\lu{y}}} \dist_2\gr{y}}.
	\]
	Note that $c$ is clearly positive due to the fact that both $\dist_1$ and $\dist_2$ are positive distributions, and that therefore $\dist_\alpha\gr{x}$ is well-defined and---because a convex mixture of positive real numbers is a positive real number---positive for all $x$ in $\stsp$.
	Furthermore, we observe that
	\begin{multline*}
		\sum_{x \in \stsp} \dist_\alpha\gr{x}
		= \sum_{\lu{x} \in \lustsp} \sum_{x \in \invlumap\gr{\lu{x}}} \dist_\alpha\gr{x} \\
		= \sum_{\lu{x} \in \lustsp} \sum_{x \in \invlumap\gr{\lu{x}}} \frac{
			\alpha \gr*{\sum_{z \in \invlumap\gr{\lumap\gr{x}}} \dist_2\gr{z}} \dist_1\gr{x}
			+ \gr{1-\alpha} \gr*{\sum_{z \in \invlumap\gr{\lumap\gr{x}}} \dist_1\gr{z}} \dist_2\gr{x}
		}{c} \\
		= \sum_{\lu{x} \in \lustsp} \sum_{x \in \invlumap\gr{\lu{x}}} \frac{
			\alpha \gr*{\sum_{z \in \invlumap\gr{\lu{x}}} \dist_2\gr{z}} \dist_1\gr{x}
			+ \gr{1-\alpha} \gr*{\sum_{z \in \invlumap\gr{\lu{x}}} \dist_1\gr{z}} \dist_2\gr{x}
		}{c}.
	\end{multline*}
	Some straightforward rearranging yields
	\begin{align*}
		\sum_{x \in \stsp} \dist_\alpha\gr{x}
		&= \frac{\alpha}{c} \gr*{\sum_{\lu{x} \in \lustsp} \gr*{\sum_{z \in \invlumap\gr{\lu{x}}} \dist_2\gr{z}} \sum_{x \in \invlumap\gr{\lu{x}}} \dist_1\gr{x}} \\
		&\qquad\qquad + \frac{1 - \alpha}{c} \gr*{\sum_{\lu{x} \in \lustsp} \gr*{\sum_{z \in \invlumap\gr{\lu{x}}} \dist_1\gr{z}} \sum_{x \in \invlumap\gr{\lu{x}}} \dist_2\gr{x}} \\
		&= \frac{1}{c} \sum_{\lu{x} \in \lustsp} \gr*{\sum_{x \in \invlumap\gr{\lu{x}}} \dist_1\gr{x}} \gr*{\sum_{z \in \invlumap\gr{\lu{x}}} \dist_2\gr{z}}
		= \frac{c}{c}
		= 1.
	\end{align*}
	We now have that $\dist_\alpha$ is a positive real-valued function on $\stsp$ with $\sum_{x \in \stsp} \dist_\alpha\gr{x} = 1$, hence $\dist_\alpha$ is indeed a positive distribution on $\stsp$.
	\qed
\end{proof}

\begin{lemma}
\label{lem:Set of lumped trms:Properties}
	Let $\trm$ be an irreducible transition rate matrix and $\lumap \colon \stsp{} \to \lustsp{}$ a lumping map.
	The associated set~$\setoflutrm{}$ of lumped transition rate matrices: (i) is non-empty and bounded, (ii) is convex and (iii) has separately specified rows.
	Furthermore, every $\lutrm$ in $\setoflutrm{}$ is irreducible, and $\norm{\lutrm} \leq \norm{\trm}$.
\end{lemma}
\begin{proof}
	We start with proving (i).
	Note that it is immediate from \eqref{eqn:setoflutrm} that $\setoflutrm$ is non-empty as $\setofposdist\gr{\stsp}$ is non-empty.
	The boundedness of $\setoflutrm$ follows from the last sentence of the stated, which we will prove last.

	We therefore move on to proving (ii).
	To that end, we fix two arbitrary elements of $\setoflutrm{}$, denoted by $\lutrm_1$ and $\lutrm_2$.
	Note that because of the way $\setoflutrm$ is constructed, there is a positive distribution $\dist_1$ ($\dist_2$) on $\stsp$ such that $\lutrm_{\dist_1} = \lutrm_1$ ($\lutrm_{\dist_2} = \lutrm_2$).
	Fix now an arbitrary $\alpha$ in the open unit interval $(0, 1)$, and let $\lutrm_{\alpha} \coloneqq \alpha \lutrm_1 + (1 - \alpha) \lutrm_2$.
	Then for all $\lu{x}$ and $\lu{y}$ in $\lustsp$,
	\begin{align*}
		&\lutrm_{\alpha}\gr{\lu{x}, \lu{y}}
		= \alpha \lutrm_1\gr{\lu{x}, \lu{y}} + \gr{1 - \alpha} \lutrm_2\gr{\lu{x}, \lu{y}} \\
		&= \sum_{x \in \invlumap\gr{\lu{x}}} \gr*{
			\alpha \gr*{\frac{\dist_1\gr{x}}{\sum_{z \in \invlumap\gr{\lu{x}}} \dist_1\gr{z}}}
			+ \gr{1 - \alpha} \gr*{\frac{\dist_2\gr{x}}{\sum_{z \in \invlumap\gr{\lu{x}}} \dist_2\gr{z}}}
		} \sum_{y \in \invlumap\gr{\lu{y}}} \trm\gr{x, y} \\
		&= \sum_{x \in \invlumap\gr{\lu{x}}} \frac{
			\alpha \gr*{\sum_{z \in \invlumap\gr{\lu{x}}} \dist_2\gr{z}} \dist_1\gr{x}
			+ \gr{1 - \alpha} \gr*{\sum_{z \in \invlumap\gr{\lu{x}}} \dist_1\gr{z}} \dist_2\gr{x}
		}{
			\gr*{\sum_{z \in \invlumap\gr{\lu{x}}} \dist_1\gr{z}} \gr*{\sum_{z \in \invlumap\gr{\lu{x}}} \dist_2\gr{z}}
		}\hspace{-5pt}\sum_{y \in \invlumap\gr{\lu{y}}} \hspace{-4pt}\trm\gr{x, y}.
	\end{align*}
	Dividing both the numerator and the denominator of the fraction in the expression above by $\sum_{\lu{z} \in \lustsp} \gr*{\sum_{z \in \invlumap\gr{\lu{z}}} \dist_1\gr{z}} \gr*{\sum_{z \in \invlumap\gr{\lu{z}}} \dist_2\gr{z}}$ yields
	\[
		\lutrm_{\alpha}\gr{\lu{x}, \lu{y}}
		= \sum_{x \in \invlumap\gr{\lu{x}}} \frac{
			\dist_\alpha\gr{x}
		}{
			\sum_{z \in \invlumap\gr{\lu{x}}} \dist_\alpha\gr{z}
		} \sum_{y \in \invlumap\gr{\lu{y}}} \trm\gr{x, y},
	\]
	where $\dist_\alpha$ is defined as in Lemma~\ref{lem:dist_alpha}.
	Since we know from Lemma~\ref{lem:dist_alpha} that $\dist_\alpha$ is a positive distribution on $\stsp$, it follows from \eqref{eqn:setoflutrm} that $\lutrm_{\alpha}$ is an element of $\setoflutrm$.
	As $\lutrm_1$, $\lutrm_2$ and $\alpha$ were arbitrary, this proves that the set $\setoflutrm$ is convex.

	Next, we prove (iii).
	To that end, we fix an arbitrary $\card{\lustsp}$-tuple $\gr{ \lutrm_{\lu{x}} \colon \lu{x} \in \lustsp }$ of which the entries---one for every state---are all elements of $\setoflutrm$.
	We know from \eqref{eqn:setoflutrm} that, for any $\lu{x}$ in $\lustsp$, there is a positive distribution $\dist_{\lu{x}}$ on $\stsp$ such that $\lutrm_{\dist_{\lu{x}}} = \lutrm_{\lu{x}}$.
	Following Definition~\ref{def:separately specified rows}, we now construct a matrix~$\lutrmst$, defined by
	\[
		\lutrmst\gr{\lu{x}, \lu{y}}
		\coloneqq \lutrm_{\lu{x}}\gr{\lu{x}, \lu{y}}
		\text{ for all } \lu{x}, \lu{y} \in \stsp.
	\]
	We need to prove that $\lutrmst$ is an element of $\setoflutrm$.
	To verify this, we observe that for all $\lu{x}$ and $\lu{y}$ in $\lustsp$,
	\[
		\lutrmst\gr{\lu{x}, \lu{y}}
		= \sum_{x \in \invlumap\gr{\lu{x}}} \frac{\dist_{\lu{x}}\gr{x}}{\sum_{z \in \invlumap\gr{\lu{x}}} \dist_{\lu{x}}\gr{z}} \sum_{y \in \invlumap\gr{\lu{y}}} \trm\gr{x, y}.
	\]
	We now divide both the numerator and the denominator of the fraction in the expression above by $\sum_{z \in \stsp} \dist_{\lumap\gr{x}}\gr{x}$, to yield
	\[
		\lutrmst\gr{\lu{x}, \lu{y}}
		= \sum_{x \in \invlumap\gr{\lu{x}}} \frac{\distst\gr{x}}{\sum_{z \in \invlumap\gr{\lu{x}}} \distst\gr{z}} \sum_{y \in \invlumap\gr{\lu{y}}} \trm\gr{x, y},
	\]
	where $\distst$ is the positive distribution---one can easily verify that this is indeed the case---on $\stsp$ defined by
	\[
		\distst\gr{x}
		\coloneqq \frac{\dist_{\lumap\gr{x}}\gr{x}}{ \sum_{z \in \stsp} \dist_{\lumap\gr{z}}\gr{z} }
		\quad\text{for all $x \in \stsp$}.
	\]
	Because this final equality holds for all $\lu{x}$ and $\lu{y}$ in $\lustsp$, we find conclude that $\lutrmst$ is indeed an element of $\setoflutrm$.

	Next, we fix an arbitrary $\lutrm$ in $\setoflutrm$.
	Let $\dist$ be the positive distribution on $\stsp$ such that $\lutrm_{\dist} = \lutrm$.
	Then $\lutrm$ is irreducible by Lemma~\ref{lem:lutrm_posdist is a trm and irreducible}.
	Furthermore,
	\begin{align*}
		\norm{\lutrm}
		&= 2 \max \set*{ \abs{\lutrm\gr{\lu{x}, \lu{x}}} \colon \lu{x} \in \lustsp } \\
		&=2 \max \set*{ \abs*{
			\sum_{x \in \invlumap\gr{\lu{x}}} \frac{\dist\gr{x}}{\sum_{z \in \invlumap\gr{\lu{x}}} \dist\gr{z}}
			\sum_{y \in \invlumap\gr{\lu{x}}} \trm\gr{x, y}
		} \colon \lu{x} \in \lustsp } \\
		&\leq 2 \max \set*{
			\sum_{x \in \invlumap\gr{\lu{x}}} \frac{\dist\gr{x}}{\sum_{z \in \invlumap\gr{\lu{x}}} \dist\gr{z}}
			\abs*{ \sum_{y \in \invlumap\gr{\lu{x}}} \trm\gr{x, y}
		} \colon \lu{x} \in \lustsp } \\
		&\leq 2 \max \set*{
			\sum_{x \in \invlumap\gr{\lu{x}}} \frac{\dist\gr{x}}{\sum_{z \in \invlumap\gr{\lu{x}}} \dist\gr{z}} \abs*{\trm\gr{x, x}}
		\colon \lu{x} \in \lustsp } \\
		&\leq 2 \max \set*{ \max\set*{ \abs{\trm\gr{x, x}} \colon x \in \invlumap\gr{\lu{x}}} \colon \lu{x} \in \lustsp } \\
		&= 2 \max \set*{ \abs{\trm\gr{x, x}} \colon x \in \stsp }
		= \norm{\trm},
	\end{align*}
	where the first and last equality follow from \eqref{eqn:Norm of trm}, the second equality follows from \eqref{eqn:lutrm_dist defintion}, the first inequality follows from the triangle inequality, the second inequality follows from the properties of transition rate matrices---i.e., non-negative off-diagonal elements and rows that sum to zero--- and where for the third inequality we use the fact that a convex combination of real numbers is always lower than the maximum of these real numbers.
	\qed
\end{proof}

	\subsection{The Lower Transition (Rate) Operator Corresponding to the Set of Lumped Transition Rate Matrices}
Since $\setoflutrm$ is non-empty and bounded by Lemma~\ref{lem:Set of lumped trms:Properties}, we know from Appendix~\ref{sapp:LTROs} that it has an associated lower transition rate operator $\lultro \colon \setoffn\gr{\lustsp} \to \setoffn\gr{\lustsp}$, defined by \eqref{eqn:ltro}:
\begin{equation}
\label{eqn:ltro for lumped}
	[\lultro \lu{f}]\gr{\lu{x}}
	= \inf \set{[\lutrm \lu{f}](\lu{x}) \colon \lutrm \in \setoflutrm{}}
	\quad\text{for all $\lu{x}$ in $\lustsp$ and all $\lu{f}$ in $\setoffn\gr{\lustsp}$.}
\end{equation}
Note that \eqref{eqn:lumped ltro}, the definition for $\lultro$ in the main text, differs from \eqref{eqn:ltro for lumped}, its proper definition.
These two definitions turn out to be equal in this case, as is stated in the following result.
\begin{proposition}
	If $\trm$ is an irreducible transition rate matrix and $\lumap \colon \stsp \to \lustsp$ a lumping map, then for all $\lu{x}$ in $\lustsp$ and $\lu{f}$ in $\setoffn\gr{\lustsp}$,
	\[
		[\lultro \lu{f}]\gr{\lu{x}}
		= \min \set*{ \sum_{\lu{y} \in \lustsp} \lu{f}\gr{\lu{y}} \sum_{y \in \invlumap\gr{\lu{y}}} Q(x, y) \colon x \in \invlumap\gr{\lu{x}} }.
	\]
\end{proposition}
\begin{proof}
	Let
	\[
		f_x
		\coloneqq \sum_{\lu{y} \in \lustsp} \lu{f}\gr{\lu{y}} \sum_{y \in \invlumap\gr{\lu{y}}} \trm\gr{x, y}
		\quad\text{for all $x \in \invlumap\gr{\lu{x}}$.}
	\]
	Then we need to prove that
	\[
		[\lultro \lu{f}]\gr{\lu{x}}
		= \min \set*{ f_x \colon x \in \invlumap\gr{\lu{x}} }.
	\]
	By combining \eqref{eqn:setoflutrm} and \eqref{eqn:ltro for lumped}, we find that
	\[
		[\lultro \lu{f}]\gr{\lu{x}}
		= \inf \set{[\lutrm \lu{f}](\lu{x}) \colon \lutrm \in \setoflutrm{}}
		=\inf \set{[\lutrm_{\dist} \lu{f}](\lu{x}) \colon \dist \in \setofposdist\gr{\stsp}}.
	\]
	Explicitly writing out the matrix-vector product $[\lutrm_{\dist} \lu{f}]\gr{\lu{x}}$ yields
	\begin{align*}
		[\lutrm_{\dist} \lu{f}]\gr{\lu{x}}
		&= \sum_{\lu{y} \in \lustsp} \lu{f}\gr{\lu{y}} \lutrm_{\dist}\gr{\lu{x}, \lu{y}} = \sum_{\lu{y} \in \lustsp} \lu{f}\gr{\lu{y}} \sum_{x \in \invlumap\gr{\lu{x}}} \frac{\dist\gr{x}}{\sum_{z \in \invlumap\gr{\lu{x}}} \dist\gr{z}} \sum_{y \in \invlumap\gr{\lu{y}}} \trm\gr{x, y} \\
		&= \sum_{x \in \invlumap\gr{\lu{x}}} \frac{\dist\gr{x}}{\sum_{z \in \invlumap\gr{\lu{x}}} \dist\gr{z}} \sum_{\lu{y} \in \lustsp} \lu{f}\gr{\lu{y}} \sum_{y \in \invlumap\gr{\lu{y}}} \trm\gr{x, y} \\
		&= \sum_{x \in \invlumap\gr{\lu{x}}} \frac{\dist\gr{x}}{\sum_{z \in \invlumap\gr{\lu{x}}} \dist\gr{z}} f_x.
	\end{align*}
	Hence, we need to prove that
	\[
		\inf \set*{\sum_{x \in \invlumap\gr{\lu{x}}} \frac{\dist\gr{x}}{\sum_{z \in \invlumap\gr{\lu{x}}} \dist\gr{z}} f_x \colon \dist \in \setofposdist\gr{\stsp}}
		= \min \set*{f_x \colon x \in \invlumap\gr{\lu{x}}}.
	\]
	Note that the right hand side is clearly a lower bound for
	\[
		\inf\set*{\sum_{x \in \invlumap\gr{\lu{x}}} \frac{\dist\gr{x}}{\sum_{z \in \invlumap\gr{\lu{x}}} \dist\gr{z}} f_x \colon \dist \in \setofposdist\gr{\stsp}}.
	\]
	We now show that it is the tightest lower bound---i.e., the infimum---of this set.
	To that end, we construct a sequence $\{ \dist_n \}_{n \in \nats}$ in $\setofposdist{}$ such that the induced sequence
	\[
		\set*{ \sum_{x \in \invlumap\gr{\lu{x}}} \frac{\dist_n\gr{x}}{\sum_{z \in \invlumap\gr{\lu{x}}} \dist_n\gr{z}} f_x }_{n \in \nats}
	\]
	converges to $\min \set*{f_x \colon x \in \invlumap\gr{\lu{x}}}$.
	Let $x^{\star}$ be an element of $\invlumap\gr{\lu{x}}$ such that $f_{x^{\star}} = \min\set*{ f_x \colon x \in \invlumap\gr{\lu{x}} }$, $c \coloneqq \frac{1}{\card{\stsp}}$ and $m \coloneqq \card{\invlumap\gr{\lu{x}}}$.
	For all $n$ in $\nats$, we define the positive distribution $\dist_n$ on $\stsp$ by
	\[
		\dist_{n}\gr{x}
		\coloneqq \begin{dcases}
			c &\text{if } x \not\in \invlumap\gr{\lu{x}} \\
			c m - \frac{c}{n} \gr{m - 1} &\text{if } x = x^{\star} \\
			\frac{c}{n} &\text{otherwise}
		\end{dcases}
		\quad\text{for all $x$ in $\stsp$.}
	\]
	Then clearly,
	\begin{multline*}
		\lim_{n \to +\infty} \sum_{x \in \invlumap\gr{\lu{x}}} \frac{\dist_n\gr{x}}{\sum_{z \in \invlumap\gr{\lu{x}}} \dist_n\gr{z}} f_x \\
		= \lim_{n \to +\infty}  \gr*{\gr*{1 - \frac{m-1}{m n}} f_{x^{\star}} + \sum_{x \in \invlumap\gr{\lu{x}} \colon x \neq x^{\star}} \frac{1}{m n} f_x}
		= f_{x^{\star}} \\
		= \min\set*{ f_x \colon x \in \invlumap\gr{\lu{x}} }.
	\end{multline*}
	\qed
\end{proof}

The following result states that $\lultro$ is irreducible, which is to be expected as it is the lower envelope of a set of irreducible transition rate matrices.
\begin{corollary}
\label{cor:lultro is ergodic}
	If $\trm$ is an irreducible transition rate matrix and $\lumap \coloneqq \stsp \to \lustsp$ a lumping map, then $\lultro$ is \emph{irreducible}.
\end{corollary}
\begin{proof}
	Recall from Lemma~\ref{lem:Set of lumped trms:Properties} that any $\lutrm$ in $\setoflutrm$ is irreducible.
	Fix now any arbitrary $\lutrm^{\star}$ in $\setoflutrm$.
	Then for any distinct $\lu{x}$ and $\lu{y}$ in $\lustsp$, there is a sequence $\lu{y} = \lu{x}_1, \dots, \lu{x}_n = \lu{x}$ in $\lustsp$ such that $\lutrm^{\star}\gr{\lu{x}_{i-1}, \lu{x}_i} > 0$ for all $i$ in $\set{1, \dots, n}$.
	By \eqref{eqn:ltro for lumped}, it then clearly holds for any $i$ in $\set{1, \dots, n}$ that
	\begin{align*}
		- [\lultro \gr{- \indica{\lu{x}_i}}]\gr{\lu{x}_{i-1}}
		&= - \inf \set*{ - \lutrm\gr{\lu{x}_{i-1}, \lu{x}_i} \colon \lutrm \in \setoflutrm} \\
		&= \sup \set*{ \lutrm\gr{\lu{x}_{i-1}, \lu{x}_i} \colon \lutrm \in \setoflutrm}
		\geq \lutrm^{\star}\gr{\lu{x}_{i-1}, \lu{x}_i} > 0.
	\end{align*}
	Consequently, $\lu{y} \upacces \lu{x}$ for any arbitrary $\lu{x}$ and $\lu{y}$, which proves the stated.
	\qed
\end{proof}

	\subsection{Laying Down the Last Pieces of the Puzzle}
For any positive and irreducible CTMC $\probm$ with initial distribution $\inidist$ and any lumping map $\lumap \colon \stsp \to \lustsp$, we define the \emph{lumped initial distribution}
\begin{equation}
\label{eqn:Lumped initial distribution}
	\luinidist{}
	\colon \lustsp \to \reals
	\colon \lu{x} \mapsto \luinidist\gr{\lu{x}}
	\coloneqq \sum_{x \in \invlumap\gr{\lu{x}}} \inidist\gr{x}.
\end{equation}
It can be immediately verified that $\luinidist$ is a positive distribution on $\lustsp$.
Furthermore, using \eqref{eqn:New defintion of lumped stochastic process:Marginal}, Lemma~\ref{lem:Lumped process:One step conditional}, \eqref{eqn:HCTMC initial} and \eqref{eqn:Lumped initial distribution} yields that
\begin{align*}
	\problum\gr{\lu{X}_0 = \lu{x}}
	&= \problum\gr{\lu{X}_0 = \lu{x} \cond \lu{X}_{\emptytseq{}} = \lu{x}_{\emptytseq{}}} \\
	&= \frac{\sum_{x \in \invlumap\gr{\lu{x}}} \probm\gr{X_0 = x \cond X_\emptytseq{} = x_{\emptytseq{}}} \probm\gr{X_\emptytseq{} = x_{\emptytseq{}}} }{\probm\gr{X_\emptytseq{} = x_{\emptytseq{}}} } \\
	&= \sum_{x \in \invlumap\gr{\lu{x}}} \probm\gr{X_0 = x}
	= \sum_{x \in \invlumap\gr{\lu{x}}} \inidist\gr{x}
	= \luinidist\gr{\lu{x}}.
\end{align*}
Hence, if we let $\lu\credset \coloneqq \set{\lu\dist_0}$, we see that the lumped stochastic process $\problum$ is \emph{consistent} with $\lu\credset$, see Appendix~\ref{sapp:Sets of consistent processes and lower expectations}.
The following intermediary result, which follows immediately from Proposition~\ref{prop:Instantaneous transition rate of the lumped stochastic process} and \eqref{eqn:setoflutrm}, states that it is also consistent with $\setoflutrm{}$; again, see Appendix~\ref{sapp:Sets of consistent processes and lower expectations}.
\begin{corollary}
\label{cor:Instantaneous transition rate of the lumped stochastic process}
	Let $\markov$ be a positive and irreducible CTMC and let $\lumap \colon \stsp \to \lustsp$ be a lumping map.
	Fix any $t$ in $\nnegreals{}$, any $u$ in $\setoftseql{t}$ and any $\lu{x}_{u}$ in $\lustsp_{u}$.
	Then there is a unique element $\lutrm_{\gr{u,\lu{x}_{u}, t}}$ of $\setoflutrm$ such that, for all $\lu{x}, \lu{y}$ in $\lustsp$,
	\begin{align}
	\label{eqn:inst trm right}
		\lim_{\Delta \to 0^{+}} \frac{\problum\gr{ \lu{X}_{t+\Delta} = \lu{y} \cond \lu{X}_{u} = \lu{x}_{u}, \lu{X}_{t} = \lu{x} } - \indica{\lu{x}}\gr{\lu{y}}}{\Delta}= \lutrm_{\gr{u,\lu{x}_{u}, t}} \gr{\lu{x}, \lu{y}}
	\intertext{and, if $t \neq 0$,}
	\label{eqn:inst trm left}
		\lim_{\Delta \to 0^{+}} \frac{\problum\gr{ \lu{X}_{t} = \lu{y} \cond \lu{X}_{u} = \lu{x}_{u}, \lu{X}_{t - \Delta} = \lu{x} } - \indica{\lu{x}}\gr{\lu{y}}}{\Delta}= \lutrm_{\gr{u,\lu{x}_{u}, t}} \gr{\lu{x}, \lu{y}}.
	\end{align}
\end{corollary}
\begin{proof}
	Follows immediately from Proposition~\ref{prop:Instantaneous transition rate of the lumped stochastic process}, \eqref{eqn:lutrm_dist defintion} and \eqref{eqn:setoflutrm}.
	\qed
\end{proof}

We now combine several of our intermediary results concerning the lumped stochastic process $\problum$ to finally end up with the result we need to prove the results in Sect.~\ref{sec:Performing inferences using the lumped process}.
\begin{corollary}
\label{cor:Lumped is well-behaved and consistent}
	If $\probm$ is a positive and irreducible CTMC and $\lumap \colon \stsp \to \lustsp$ a lumping map, then the associated lumped stochastic process~$\problum$ is contained in $\setofconsproc_{\inidist, \trm, \lumap} \coloneqq \setofconsproc_{\setoflutrm, \lu\credset{}}^{\mathrm{W}}$.
\end{corollary}
\begin{proof}
	Recall that $\problum$ is well-behaved by Corollary~\ref{cor:Lumped process is well-behaved}.
	Furthermore, as we have just seen in this section, $\problum$ is consistent with $\setoflutrm$ and $\lu\credset$.
	The stated now follows because, by definition, $\setofconsproc_{\inidist, \trm, \lumap}$ contains all well-behaved stochastic processes that are consistent with $\setoflutrm$ and $\lu\credset$.
	\qed
\end{proof}

	\section{Proofs of the Results in Sect.~\ref{sec:Performing inferences using the lumped process}}
\label{app:Proofs of main results}
In the main text, we limited ourselves to determining bounds on marginal and limit expectations of functions $f$ in $\setoffn\gr{\stsp}$ that are lumpable with respect to $\lumap$, mainly due to length constraints.
Since this length constraint is not present in this extended pre-print, we here drop this restriction.

Let $\lumap \colon \stsp \to \lustsp$ be a lumping map.
Then the reduction to $\lustsp$ of a non-lumpable $f$ in $\setoffn\gr{\stsp}$ is not unequivocally defined.
Two restrictions that will turn out to be useful in our setting are $\lu{f}_L$ and $\lu{f}_U$ in $\setoffn\gr{\lustsp}$, defined for all $\lu{x}$ in $\lustsp$ as
\[
	\lu{f}_L\gr{\lu{x}}
	\coloneqq \min \set{ f\gr{x} \colon x \in \invlumap\gr{\lu{x}} }
	\quad\text{and}\quad
	\lu{f}_U\gr{\lu{x}}
	\coloneqq \max \set{ f\gr{x} \colon x \in \invlumap\gr{\lu{x}} }.
\]
Note that if $f$ is lumpable with respect to $\lumap$, then $\lu{f}_L = \lu{f} = \lu{f}_U$.
Moreover, we have the following two properties.
\begin{lemma}
\label{lem:Bounds on function by lumping}
	If $\lumap \colon \stsp \to \lustsp$ is a lumping map, then for all $x$ in $\stsp$ and $f$ in $\setoffn\gr{\stsp}$,
	\[
		\lu{f}_L\gr{\lumap\gr{x}}
		\leq f\gr{x}
		\leq \lu{f}_U\gr{\lumap\gr{x}}.
	\]
\end{lemma}
\begin{proof}
	Follows immediately from the definition of $\lu{f}_L$ and $\lu{f}_U$.
	\qed
\end{proof}
\begin{lemma}
\label{lem:Lumped expectation as bound on marginal expectation}
	Let $\probm$ be a positive and irreducible CTMC and $\lumap \colon \stsp \to \lustsp$ a lumping map.
	Then for all $f$ in $\setoffn\gr{\stsp}$ and $t$ in $\nnegreals$,
	\[
		\prevlum\gr{\lu{f}_L\gr{\lu{X}_t}}
		\leq \prevm\gr{f\gr{X_t}}
		\leq \prevlum\gr{\lu{f}_U\gr{\lu{X}_t}}.
	\]
\end{lemma}
\begin{proof}
	We start by proving the lower bound.
	Note that by Lemma~\ref{lem:Bounds on function by lumping} and the monotonicity of $\prevm$,
	\[
		\prevm\gr{\lu{f}_L\gr{\lumap\gr{X_t}}}
		\leq \prevm\gr{f\gr{X_t}}.
	\]
	Some straightforward manipulations yield
	\begin{align*}
		\prevm\gr{\lu{f}_L\gr{\lumap\gr{X_t}}}
		&= \sum_{x \in \stsp} \lu{f}_L\gr{\lumap\gr{x}} \probm\gr{X_t = x}
		= \sum_{\lu{x} \in \lustsp} \lu{f}_L\gr{\lu{x}} \sum_{x \in \invlumap\gr{\lu{x}}} \probm\gr{X_t = x} \\
		&= \sum_{\lu{x} \in \lustsp} \lu{f}_L\gr{\lu{x}} \probm\gr*{\bigcup_{x \in \invlumap\gr{\lu{x}}} \gr{X_t = x}}
		= \sum_{\lu{x} \in \lustsp} \lu{f}_L\gr{\lu{x}} \problum\gr{\lu{X}_t = \lu{x}} \\
		&= \prevlum\gr{\lu{f}_L\gr{\lu{X}_t}},
	\end{align*}
	where the fourth equality follows from \eqref{eqn:invlumap_paths:state instantiation} and \eqref{eqn:New defintion of lumped stochastic process:Marginal}.
	Combining this equality with the previously obtained inequality immediately yields the lower bound of the stated.

	To prove the upper bound, we apply the lower bound on the function $g \coloneqq -f$.
	Note that $\lu{g}_L = - \lu{f}_U$, and that
	\[
		\prevlum\gr{- \lu{f}_U\gr{\lu{X}_t}}
		= \prevlum\gr{\lu{g}_L\gr{\lu{X}_t}}
		\leq \prevm\gr{g\gr{X_t}}
		= \prevm\gr{- f\gr{X_t}}
	\]
	clearly implies that $\prevm\gr{f\gr{X_t}} \leq \prevlum\gr{\lu{f}_U\gr{\lu{X}_t}}$.
	\qed
\end{proof}

The following result is slightly more general than Theorem~\ref{the:Bounds on marginal expectation}.
Recall from Sect.~\ref{ssec:Induced imprecise CTMC} that we use $\lulto_t$ to denote the lower transition operator over $t$ associated with $\lultro$ according to \eqref{eqn:lumped lower transition operator}.
\begin{proposition}
\label{prop:General bounds on marginal expectation}
	If $\probm$ is a positive and irreducible CTMC and $\lumap \colon \stsp \to \lustsp$ a lumping map, then for all $f$ in $\setoffn\gr{\stsp}$ and $t$ in $\nnegreals$,
	\[
		\luinidist \lulto_t \lu{f}_L
		\leq \prev\gr{f\gr{X_t}}
		= \inidist \tm_t f
		\leq - \luinidist \lulto_t \gr{- \lu{f}_U}.
	\]
\end{proposition}
\begin{proof}
	We start by proving the lower bound.
	By Lemma~\eqref{lem:Lumped expectation as bound on marginal expectation},
	\[
		\prevlum\gr{\lu{f}_L\gr{\lu{X}_t}}
		\leq \prevm\gr{f\gr{X_t}}.
	\]
	It follows from Corollary~\ref{cor:Lumped is well-behaved and consistent} that
	\[
		\underline{\prev}_{\setoflutrm{}, \lu\credset}^{\mathrm{M}} \gr{\lu{f}_L\gr{\lu{X}_t}}
		\leq \prevlum\gr{\lu{f}_L\gr{\lu{X}_t}}.
	\]
	Moreover, from Proposition~\ref{prop:Imprecise law of iterated expectation} and Proposition~\ref{prop:Imprecise Markov property} (in that order)---which we may both use due to Lemma~\ref{lem:Set of lumped trms:Properties}---it follows that
	\[
		\underline{\prev}_{\setoflutrm{}, \lu\credset}^{\mathrm{M}} \gr{\lu{f}_L\gr{\lu{X}_t}}
		= \underline{\prev}_{\setoflutrm{}, \lu\credset}^{\mathrm{M}} \gr{ \underline{\prev}_{\setoflutrm{}, \lu\credset}^{\mathrm{M}} \gr{\lu{f}_L\gr{\lu{X}_t} \cond \lu{X}_0} }
		= \underline{\prev}_{\setoflutrm{}, \lu\credset}^{\mathrm{M}} \gr{ [\lulto_t \lu{f}_L]\gr{\lu{X}_0} }.
	\]
	Since $\lu\credset$ is a singleton, it follows from \cite[Proposition~9.3]{2017KrakDeBock} that this can be rewritten as
	\begin{align*}
		\underline{\prev}_{\setoflutrm{}, \lu\credset}^{\mathrm{M}} \gr{\lu{f}_L\gr{\lu{X}_t}}
		%&= \inf_{\problum' \in \mathbb{P}_{\setoflutrm{}, \lu\credset}^{\mathrm{M}}} \sum_{\lu{x} \in \lustsp} \problum'\gr{\lu{X}_0 = \lu{x}} [\lulto_t \lu{f}_L]\gr{\lu{x}}
		= \sum_{\lu{x} \in \lustsp} \luinidist\gr{\lu{x}} [\lulto_t \lu{f}_L]\gr{\lu{x}} %\\
		%&
		= \luinidist \lulto_t \lu{f}_L.
	\end{align*}
	Finally, combining all that we have found so far yields the lower bound of the stated:
	\[
		\luinidist \lulto_t \lu{f}_L
		= \underline{\prev}_{\setoflutrm{}, \lu\credset}^{\mathrm{M}} \gr{\lu{f}_L\gr{\lu{X}_t}}
		\leq \prevlum\gr{\lu{f}_L\gr{\lu{X}_t}}
		= \prevm\gr{\lu{f}_L\gr{\lumap\gr{X_t}}}
		\leq \prevm\gr{f\gr{X_t}}.
	\]

	To prove the upper bound, we simply apply the lower bound to $g \coloneqq -f$.
	This yields
	\[
		\luinidist \lulto_t \gr{- \lu{f}_U}
		= \luinidist \lulto_t \lu{g}_L
		\leq \prevm\gr{g\gr{X_t}}
		= \prevm\gr{- f\gr{X_t}},
	\]
	which clearly implies the upper bound of the stated.
	\qed
\end{proof}

\begin{proofof}{Theorem~\ref{the:Bounds on marginal expectation}}
	Since $f$ is lumpable with respect to $\lumap$, we know that $\lu{f}_L = \lu{f} = \lu{f}_U$.
	Therefore, the stated follows immediately from Proposition~\ref{prop:General bounds on marginal expectation}.
	\qed
\end{proofof}

\begin{proposition}
\label{prop:More general bounds on limit expectation}
	If $\prob$ is a CTMC with irreducible transition rate matrix $\trm$ and $\lumap \colon \stsp \to \lustsp$ a lumping map, then for all $f$ in $\setoffn\gr{\stsp}$, $\delta$ in $\posreals$ such that $\delta \norm{\trm} < 2$ and $n$ in $\natz$,
	\[
		\min \gr{I + \delta \lultro}^n \lu{f}_L
		\leq \limprevm\gr{f}
		\leq - \min \gr{I + \delta \lultro}^n (- \lu{f}_U).
	\]
	Furthermore, for fixed $\delta$, the lower and upper bounds in this expression become monotonously tighter with increasing $n$, and each converges to a (possibly different) constant as $n$ approaches $+ \infty$.
\end{proposition}
\begin{proof}
	As $\trm$ is irreducible, we know from Sect.~\ref{ssec:Irreducibility of homogeneous CTMCs} that there is a unique positive distribution, denoted by $\limdist$, such that $\limdist \trm = 0$.
	Hence, for all $y$ in $\stsp$,
	\[
		\sum_{x \in \stsp} \limdist\gr{x} \trm\gr{x, y}
		= 0.
	\]
	% Furthermore, we know that for all $n$ in $\nats$, $\delta$ in $\nnegreals$ such that $\delta \norm{\trm} \leq 2$ and $f$ in $\setoffn\gr{\stsp}$,
	% \[
	% 	\min \gr{I + \delta \trm}^n f
	% 	\leq \limprevm\gr{f}
	% 	= \limdist f
	% 	\leq \max \gr{I + \delta \trm}^n f.
	% \]

	Consider now $\lutrm_{\infty} \coloneqq \lutrm_{\limdist}$.
	Then by Lemma~\ref{lem:Set of lumped trms:Properties}, $\lutrm_{\infty}$ is irreducible.
	Let $\lulimdist'$ be the unique positive distribution on $\lustsp$ that satisfies the equilibrium condition $\lulimdist' \lutrm_{\infty} = 0$.
	We now claim that $\lulimdist' = \lulimdist$, where $\lulimdist$ is the positive distribution on $\lustsp$ defined by
	\[
		\lulimdist\gr{\lu{x}}
		= \sum_{x \in \invlumap\gr{\lu{x}}} \limdist\gr{x}
		\quad\text{for all $\lu{x}$ in $\lustsp$.}
	\]
	To verify this claim, we fix any $\lu{y}$ in $\lustsp$ and see that
	\begin{align*}
		\sum_{\lu{x} \in \lustsp} \lulimdist\gr{\lu{x}} \lutrm_{\infty}\gr{\lu{x}, \lu{y}}
		&= \sum_{\lu{x} \in \lustsp} \lulimdist\gr{\lu{x}} \sum_{x \in \invlumap\gr{\lu{x}}} \frac{\limdist\gr{x}}{\sum_{z \in \invlumap\gr{\lu{z}}} \limdist\gr{z}} \sum_{y \in \invlumap\gr{\lu{y}}} \trm\gr{x, y} \\
		&= \sum_{\lu{x} \in \lustsp} \gr*{\sum_{x \in \invlumap\gr{\lu{x}}} \limdist\gr{x}} \sum_{x \in \invlumap\gr{\lu{x}}} \frac{\limdist\gr{x}}{\sum_{z \in \invlumap\gr{\lu{z}}} \limdist\gr{z}} \sum_{y \in \invlumap\gr{\lu{y}}} \trm\gr{x, y} \\
		&= \sum_{\lu{x} \in \lustsp} \sum_{x \in \invlumap\gr{\lu{x}}} \limdist\gr{x} \sum_{y \in \invlumap\gr{\lu{y}}} \trm\gr{x, y} \\
		&= \sum_{y \in \invlumap\gr{\lu{y}}} \sum_{x \in \stsp} \limdist\gr{x} \trm\gr{x, y}
		= 0.
	\end{align*}
	As $\lu{y}$ was arbitrary, we find that $\lulimdist$ satisfies the equilibrium condition $\lulimdist \lutrm_{\infty} = 0$.
	Since $\lulimdist'$ is the unique positive distribution that satisfies this equilibrium condition, we conclude that $\lulimdist = \lulimdist'$.

	Fix now any $f$ in $\setoffn\gr{\stsp}$, $\delta$ in $\posreals$ such that $\delta \norm{\trm} < 2$ and $n$ in $\natz$.
	Note that if $n = 0$, the stated trivially holds.
	Hence, we now consider the case $n > 0$, starting with the lower bound.
	Recall from Lemma~\ref{lem:Set of lumped trms:Properties} that $\norm{\lutrm_\infty} \leq \norm{\trm}$, such that $\delta \norm{\lutrm_\infty} < 2$.
	Hence, from Lemma~\ref{lem:Bounds on function by lumping} it follows that
	\[
		\lulimdist \lu{f}_L
		\leq \limprevm\gr{f}
		= \limdist f.
	\]
	As $\lutrm_\infty$ is irreducible with stationary distribution $\lulimdist$, it follows from Lemma~\ref{lem:Lower and upper bound for limit expectation} that
	\[
		\min \gr{I + \delta \lutrm_{\infty}}^n \lu{f}_L
		\leq \lulimdist \lu{f}_L.
	\]
	Since $\lutrm_{\infty}$ is an element of $\setoflutrm{}$ by construction, it follows from \eqref{eqn:ltro for lumped} that $\lultro \lu{g} \leq \lutrm_{\infty} \lu{g}$ for any $\lu{g}$ in $\setoffn\gr{\lustsp}$, which implies that
	\begin{equation}
	\label{eqn:monotonicity used in proof}
		\gr{I + \delta \lultro} \lu{g}
		\leq \gr{I + \delta \lutrm_{\infty}} \lu{g}.
	\end{equation}
	Clearly, \eqref{eqn:monotonicity used in proof} implies that $\gr{I + \delta \lultro} f_L \leq \gr{I + \delta \lutrm_{\infty}} f_L$.
	In case $n = 1$, this is sufficient to prove the lower bound.
	In case $n > 1$, we need some more properties.
	Since $\delta \norm{\lutrm_\infty} < 2$, it follows from Lemma~\ref{lem:I + delta trm is tm} that $\gr{I + \delta \lutrm_\infty}$ is a transition matrix.
	Consequently, repeated application of \eqref{eqn:monotonicity used in proof} and the monotonicity of this transition matrix yields
	\begin{multline*}
		\gr{I + \delta \lutrm_{\infty}}^n f_L
		= \gr{I + \delta \lutrm_{\infty}}^{n-1} \gr{I + \delta \lutrm_{\infty}} f_L
		\geq \gr{I + \delta \lutrm_{\infty}}^{n-1} \gr{I + \delta \lultro} \lu{f}_L \\
		\geq \cdots
		\geq \gr{I + \delta \lutrm_{\infty}} \gr{I + \delta \lultro}^{n-1} \lu{f}_L
		\geq \gr{I + \delta \lultro}^n \lu{f}_L.
	\end{multline*}
	Combining all intermediate results, we find that
	\[
		\min \gr{I + \delta \lultro}^n \lu{f}_L
		\leq \min \gr{I + \delta \lutrm_{\infty}}^n \lu{f}_L
		\leq \lulimdist \lu{f}_L
		\leq \limdist f
		= \limprevm\gr{f},
	\]
	which proves the lower bound of the stated.

	The upper bound follows from applying the lower bound to $g \coloneqq -f$.
	As $\lu{g}_L = - \lu{f}_U$, we find that
	\[
		\min \gr{I + \delta \lultro}^n \gr{-\lu{f}_U}
		= \min \gr{I + \delta \lultro}^n \gr{\lu{g}_L}
		\leq \limprevm\gr{g}
		= \limprevm\gr{-f}.
	\]
	The upper bound now follows immediately from this inequality:
	\[
		\limprevm\gr{f}
		= -\limprevm\gr{-f}
		\leq - \min \gr{I + \delta \lultro}^n \gr{-\lu{f}_U}.
	\]

	We end this proof by verifying the statement concerning the monotonous convergence of the lower bound.
	Observe that by Lemma's \ref{lem:Norm of ltro} and \ref{lem:Set of lumped trms:Properties},
	\[
		\norm{\lultro}
		= \sup \set*{\norm{\lutrm} \colon \lutrm \in \setoflutrm{}}
		\leq \norm{\trm}.
	\]
	Hence, since $\delta \norm{\trm} < 2$, we find that $\delta \norm{\lultro} < 2$.
	The monotonous convergence now follows immediately from Corollaries \ref{cor:lto skeleton of irreducible ltro converges} and \ref{cor:lultro is ergodic}.
	\qed
\end{proof}

\begin{proofof}{Theorem~\ref{the:iterative approximation}}
	We know from \eqref{eqn:Norm of trm} that $\norm{\trm} = 2 \max \set{\abs{\trm\gr{x,x}} \colon x \in \stsp}$.
	Furthermore, since $f$ is lumpable with respect to $\lumap$, we know that $\smash{\lu{f}_L = \lu{f} = \lu{f}_U}$.
	Therefore, the stated follows immediately from Proposition~\ref{prop:More general bounds on limit expectation}.
	\qed
\end{proofof}
}{}

\end{document}